\documentclass{article}

\title{A variant of Mathias forcing that preserves~\ACA}
\author{Fran{\c c}ois G. Dorais}
\date{August 15, 2010\\{\small(Revised April 25, 2012)}}
\usepackage{enumerate}
\usepackage{amsmath}
\usepackage{amssymb}
\usepackage{amsthm}
\usepackage{hyperref}

\newcounter{counter}[section]

\theoremstyle{plain}
\newtheorem{corollary}[counter]{Corollary}
\newtheorem{lemma}[counter]{Lemma}
\newtheorem{proposition}[counter]{Proposition}
\newtheorem{theorem}[counter]{Theorem}

\theoremstyle{definition}
\newtheorem{definition}[counter]{Definition}

\newcommand{\N}{\mathbb{N}}

\newcommand{\R}{\mathbb{R}}

\newcommand{\MN}{\mathcal{N}}
\newcommand{\FN}{\mathcal{N}}

\renewcommand{\models}{\vDash}

\newcommand{\forces}{\Vdash}

\newcommand{\Sforces}{\Vdash_{S}}

\newcommand{\Hforces}{\Vdash_{H}}

\newcommand{\starforces}{\blacktriangleright}
\newcommand{\nstarforces}{\not\blacktriangleright}

\newcommand{\pow}[1]{\mathcal{P}(#1)}
\newcommand{\fin}[1]{\mathcal{P}_{<\infty}(#1)}

\newcommand{\tree}[1]{U(#1)}
\newcommand{\nbhd}[1]{[#1]}

\DeclareMathOperator{\dom}{dom}
\mathchardef\mhyphen="2D

\newcommand{\sub}{\mathop{\dot{\smash{-}}}}
\newcommand{\res}{{\upharpoonright}}
\newcommand{\seq}[1]{(#1)}
\newcommand{\set}[1]{\lbrace#1\rbrace}
\newcommand{\pair}[2]{\langle#1,#2\rangle}
\newcommand{\fst}[1]{\mathsf{1st}(#1)}
\newcommand{\snd}[1]{\mathsf{2nd}(#1)}

\newcommand{\FID}{\mathsf{Fin}}

\newcommand{\ACA}{\ensuremath{\mathsf{ACA}_0}}
\newcommand{\RCA}{\ensuremath{\mathsf{RCA}_0}}

\newcommand{\WKL}{\ensuremath{\mathsf{WKL}_0}}
\newcommand{\COH}{\ensuremath{\mathsf{COH}}}

\newcommand{\Ind}[1]{\ensuremath{\mathsf{I}{#1}}}
\newcommand{\Bnd}[1]{\ensuremath{\mathsf{B}{#1}}}

\newcommand{\lthen}{\rightarrow}
\newcommand{\liff}{\leftrightarrow}
\newcommand{\THEN}{\Longrightarrow}
\newcommand{\IFF}{\Longleftrightarrow}

\begin{document}

\maketitle

\begin{abstract}\noindent
  We present and analyze \(F_\sigma\)-Mathias forcing, which is similar but tamer than Mathias forcing. In particular, we show that this forcing preserves certain weak subsystems of second-order arithmetic such as \ACA\ and \WKL\ + \Ind{\Sigma^0_2}, whereas Mathias forcing does not. We also show that the needed reals for \(F_\sigma\)-Mathias forcing (in the sense of Blass~\cite{Blass}) are just the computable reals, as opposed to the hyperarithmetic reals for Mathias forcing.
\end{abstract}

\section*{Introduction}

Mathias forcing has recently received much attention in the reverse mathematics community for its use in the analysis of Ramsey's Theorem for pairs in subsystems of second-order arithmetic. Using a variant of Mathias forcing, Cholak, Jockusch, and Slaman~\cite{CholakJockuschSlaman} have shown that \(\mathsf{RT}^2_2\) is \(\Pi^1_1\)-conservative over \RCA\ + \Ind{\Sigma^0_2}. They have also shown that every computable coloring \(c:[\N]^2\to\set{0,1}\) has a low-2 homogeneous set. Similar methods have also been used by Dzhafarov and Jockusch~\cite{DzhafarovJockusch} to reprove a result of Seetapun~\cite{SeetapunSlaman} that every computable coloring \(c:[\N]^2\to\set{0,1}\) has a cone avoiding homogeneous set.

However, Blass~\cite{Blass} has shown that if \(G\) is a Mathias generic over a model of \(\mathsf{ZFC},\) then \(G\) computes all hyperarithmetic reals. Consequently, in order to use Mathias forcing in weak subsystems of second-order arithmetic, one must jump through several hoops in order to prevent the generic real from being too close to a true generic. The main purpose of the present paper to remedy this situation by introducing \(F_\sigma\)-Mathias forcing, which is a replacement for Mathias forcing that can safely be used over weak systems of second-order arithmetic. Unlike Mathias reals which compute every hyperarithmetic real, we show that \(F_\sigma\)-Mathias forcing can avoid computing any non-computable real. We also show that \(F_\sigma\)-Mathias forcing preserves \ACA\ as well as \WKL\ + \Ind{\Sigma^0_2}.

The plan of the paper is as follows:
\begin{itemize}
\item In Section~\ref{S:Submeasures} we develop the combinatorics of \(F_\sigma\)-ideals which are necessary to define \(F_\sigma\)-Mathias forcing. Inspired by work of Mazur~\cite{Mazur}, we will use lower semicontinuous submeasures to code \(F_\sigma\)-ideals. An interesting side result of this section is that under \WKL\ + \Bnd{\Sigma^0_2}, every free \(F_\sigma\)-ideal admits such a representation (Theorem~\ref{T:Mazur}).
\item In Section~\ref{S:Forcing}, we define \(F_\sigma\)-Mathias forcing and the associated forcing language and forcing relation. Although our definitions are specialized to \(F_\sigma\)-Mathias forcing, our methodology is very general and can be used to define the forcing language and forcing relation for a wide variety of forcings for adding a real.
\item Section~\ref{S:Witnessing} is the core of the paper. We will establish a series of \emph{witnessing theorems} for the forcing relation (Theorems~\ref{T:Pi2Witnessing} and~\ref{T:Pi3Witnessing}). These are results of the form if a condition forces a statement, then there is an extension of this condition that forces the Skolemization of the statement in question. Such results are instrumental to prove conservation results.
\item In Section~\ref{S:Extension}, we define the forcing extension and we establish preservation theorems for weak subsystems of second-order arithmetic. In particular, we show that if the ground model satisfies \ACA\ or \WKL\ + \Ind{\Sigma^0_2}, then so does the generic extension (Theorems~\ref{T:PresACA} and~\ref{T:PresIS2}, respectively). Partial results for weaker subsystems of second-order arithmetic are also presented (Propositions~\ref{P:PresRCA} and~\ref{P:PresWKL}).
\item In Section~\ref{S:Applications}, we present a few applications of \(F_\sigma\)-Mathias forcing. We first show that \(F_\sigma\)-Mathias generic sets are cohesive for sets in the ground model. Then we show that \(F_\sigma\)-Mathias generic sets can be forced to avoid cones. This extends some results of Dzhafarov and Jockusch~\cite{DzhafarovJockusch}.
\end{itemize}

\noindent
For the remainder of this section, we will present some background and conventions for the paper.

For the purpose of forcing, we will find it convenient to use a functional interpretation of the basic systems \RCA\ and \ACA. Our basic structures are of the form \(\MN = (\N,\FN_1,\FN_2,\dots)\) where \(\N\) is the underlying set and each \(\FN_k\) is a set of functions \(\N^k\to\N\) which together form an algebraic clone: each \(\FN_k\) contains all the constant functions, the projections \(\pi_i(x_1,\dots,x_k) = x_i,\) and if \(f \in \FN_\ell\) and \(g_1,\dots,g_\ell \in \FN_k\) then the superposition \(f(g_1(x_1,\dots,x_k),\dots,g_\ell(x_1,\dots,x_k))\) belongs to \(\FN_k.\) For convenience, we will often think of elements of \(\N\) as nullary functions and we will write \(\FN_0\) instead of \(\N\) when appropriate.

On top of this basic structure, we require closure under \emph{primitive recursion}: there are distinguished \(0 \in \N\) (zero) and \(\sigma \in \FN_1\) (successor) such that for any \(f \in \FN_{k-1}\) and \(g \in \FN_{k+1}\) there is a unique \(h \in \FN_k\) such that \[h(0,\bar{w}) = f(\bar{w}) \quad\mbox{and}\quad h(\sigma(x),\bar{w}) = g(h(x,\bar{w}),x,\bar{w})\] for all \(x, \bar{w} \in \N.\) Note that the uniqueness requirement on \(h\) is crucial since this is the only form of induction in our system. 

Using primitive recursion, we can define the usual arithmetic operations such as addition, multiplication, truncated subtraction (\(x\sub y = \max(x-y,0)\)) together with the usual identities between them. \[\begin{array}{c@{\qquad}c}
  x + (y + z) = (x + y) + z & x + y = y + x  \\ 
  x \cdot (y \cdot z) = (x \cdot y) \cdot z & x \cdot y = y \cdot x \\
  x + 0 = x = 0 + x & x \cdot 1 = x = 1 \cdot x \\
  (x + y) \sub y = x & x \sub (y + z) = (x \sub y) \sub z \\
  x\cdot(y+z) = x\cdot y + x \cdot z & x \cdot (y \sub z) = x \cdot z \sub y \cdot z \\
  x + (y \sub x) = y + (x \sub y) & x \sub (x \sub y) = y \sub (y \sub x)
\end{array}\]
Finally, we will assume the \emph{dichotomy axiom} \[x \sub y = 0 \lor y \sub x = 0,\] which is necessary to show that the relation \(x \leq y\) defined by \(x \sub y = 0\) is a linear ordering of \(\N.\)

Atop the basic axioms described above, we will consider two second-order axioms.
\begin{description}
\item[\normalfont\emph{Uniformization axiom}]\mbox{}\\ For every \(f \in \FN_{k+1}\) such that \(\forall \bar{w}\,\exists x\,{f(x,\bar{w}) = 0},\) there is a \(g \in \FN_k\) such that \(\forall\bar{w}\,{f(g(\bar{w}),\bar{w}) = 0}.\)
\item[\normalfont\emph{Minimization axiom}]\mbox{}\\ For every \(f \in \FN_{k+1}\) there is a \(g \in \FN_k\) such that \(\forall x,\bar{w}\,f(x,\bar{w}) \geq f(g(\bar{w}),\bar{w}).\)
\end{description}
Note that minimization implies uniformization. Uniformization ensures the existence of all general recursive functions; minimization ensures the existence of arithmetically defined functions.

Every functional structure \(\MN\) corresponds to a set-based structure \((\N;\pow{\N};0,1,{+},{\cdot})\) for second-order arithmetic as described in~\cite{Simpson}, where \(\pow{\N}\) consists of all subsets of \(\N\) whose characteristic function is in \(\FN_1.\) The latter structure is a model of \RCA\ if and only if the uniformization axiom holds in \(\MN\); it is a model of \ACA\ if and only if the minimization axiom holds in \(\MN.\) Conversely, given a traditional model \((\N;\pow{\N};0,1,{+},{\cdot})\) of \RCA, we can define \(\FN_k\) to be the class of all functions \(\N^k\to\N\) whose coded graph belongs to \(\pow{\N}\) and the resulting structure is a functional model which satisfies uniformization; a traditional model of \ACA\ similarly corresponds to a functional model which satisfies minimization. Since our choice to adopt functional models is a matter of convenience, we will freely use this translation between functional models and traditional models.

The fact that our basic axioms together with uniformization correspond to \RCA\ was observed by Kohlenbach~\cite{Kohlenbach} (where uniformization is denoted \(\mathrm{QF}\mhyphen\mathrm{AC}^{0,0}\)). Hirschfeldt and Shore noticed what is essentially the same fact in~\cite[Proposition~6.6]{HirschfeldtShore}. The fact that minimization corresponds to \ACA\ can be seen by using it to compute Turing jumps. 

For the remainder of this paper, we work inside a functional model \(\MN.\) Every result has in parentheses the assumptions that the model \(\MN\) needs to satisfy in order for the result to hold. For example, Theorem~\ref{T:Mazur} says that if \(\MN \models \WKL + \Bnd{\Sigma^0_2}\) then every free \(F_\sigma\)-ideal coded in \(\MN\) is the ideal of finite sets for some integer-valued lower semicontinuous sumbmeasure coded in \(\MN.\) To avoid confusion, we will use the term \emph{set} exclusively for collections of first-order objects, and the term \emph{class} exclusively for collections of second-order objects. Internal sets are identified with their characteristic functions.

\section{Submeasures and Free \(F_\sigma\)-Ideals}\label{S:Submeasures}

In this section, we will show how to recast and utilize the classical combinatorial concepts of \(F_\sigma\)-ideals and lower semicontinuous submeasures in second-order arithmetic. For this purpose, we will initially use classical set-theoretic terminology to discuss these objects. Our terminology will be classical for the most part, but the reader should keep in mind that we are working inside a functional model \(\MN.\) For example, the internal powerclass \(\pow{\N}\) should be understood to be the class of functions $\N\to\set{0,1}$ in \(\FN_1.\) Of course, these functions are identified with the subsets of \(\N\) that they characterize and they will be handled that way.

\begin{definition}\label{D:ID}
  A class \(\mathcal{J} \subseteq \pow{\N}\) is a \emph{free ideal} when it satisfies the following three conditions.
  \begin{enumerate}[\upshape(i)]
  \item\label{ID:mono} 
    If \(X \subseteq Y\) and \(Y \in \mathcal{J}\) then \(X \in \mathcal{J}.\)
  \item\label{ID:union} 
    If \(X_0,\dots,X_k \in \mathcal{J}\) (\(k \in \N\)) then \(X_0 \cup \cdots \cup X_k \in \mathcal{J}.\)
  \item\label{ID:free}
    For every \(n \in \N,\) \(\set{0,1,\dots,n} \in \mathcal{J}.\)
  \end{enumerate}
\end{definition}

\noindent
Condition~\eqref{ID:union} is to be understood as requiring \(\mathcal{J}\) to be closed under all internally finite unions, not just the truly finite ones. 

The smallest possible free ideal is the class of all internally finite sets. However, this class satisfies condition~\eqref{ID:union} only when \Bnd{\Sigma^0_2} holds. This is a general phenomenon for \(F_\sigma\)-ideals: \Bnd{\Sigma^0_2} is necessary to show that they are closed under internally finite unions. In the reverse direction, the Weak K{\"o}nig Lemma is often necessary to show that certain sets are internally finite unions of smaller sets. For these reasons, our base theory will generally be \WKL\ + \Bnd{\Sigma^0_2}, which is the minimum necessary to develop a sound theory of free \(F_\sigma\)-ideals.

A convenient way to encode free \(F_\sigma\)-ideals is via lower semicontinuous submeasures. Recall that a submeasure is a map \(\mu:\pow{\N}\to[0,\infty]\) such that \(\mu(\varnothing) = 0\) and \[\mu(X) \leq \mu(X \cup Y) \leq \mu(X) + \mu(Y)\] for all \(X, Y \subseteq \N.\) This map \(\mu\) is lower semicontinuous if the preimages \(\mu^{-1}[0,a]\) are all closed classes in \(\pow{\N}\) (endowed with the usual product topology). It follows that the ideal of \(\mu\)-finite sets \[\FID(\mu) = \bigcup_{n=0}^{\infty} \mu^{-1}[0,n]\] is an \(F_\sigma\)-ideal, and this ideal is free when \(\mu(\set{x}) < \infty\) for every \(x \in \N.\) 

Another convenient property of lower semicontinuous submeasures is that they are completely determined by their values on finite sets. Indeed, we always have
\begin{equation}\label{eq:sublimit}
  \mu(X) = \sup_{n \in \N} \mu(X \cap \set{0,\dots,n-1}).
\end{equation}
This allows us to code lower semicontinuous submeasures in second-order arithmetic.
In the following, we will use \(\fin{\N}\) to denote the set of all codes for internally finite sets (the encoding is immaterial so long as the basic operations are primitive recursive). We will sometimes abuse notation and identify finite sets, which are second-order objects, with their codes, which are first-order objects.

\begin{definition}
  A \emph{code for a} (\emph{lower semicontinuous}) \emph{submeasure} is a function \(\mu:\fin{\N}\to\R\) such that \(\mu(\varnothing) = 0\) and \[\max(\mu(x),\mu(y)) \leq \mu(x \cup y) \leq \mu(x) + \mu(y)\] for all \(x, y \in \fin{\N}.\) In other words, \(\mu:\fin{\N}\to\R\) is a monotone and subadditive function such that \(\mu(\varnothing) = 0.\)
\end{definition}

\noindent
Such a code naturally extends to a lower semicontinuous submeasure \(\mu\) defined by~\eqref{eq:sublimit}. This makes perfect sense in \ACA, but the supremum in~\eqref{eq:sublimit} does not necessarily exist in weaker systems. Nevertheless, one can always make sense of inequalities of the form \(\mu(X) \leq r,\) for any \(r \in [0,\infty],\) via \[\mu(X) \leq r \liff \forall n\,{\mu(X \cap \set{0,\dots,n-1}) \leq r}.\] Similar interpretations can be found for all other types of inequalities. In particular, \[\mu(X) < \infty \liff \exists m\,\forall n\,{\mu(X \cap \set{0,\dots,n-1}) \leq m}\] which allows us to define the class \(\FID(\mu)\) of \(\mu\)-finite sets even in \RCA.

\begin{proposition}[\RCA\ + \Bnd{\Sigma^0_2}]\label{P:Freeness}
  If \(\mu\) is a \textup{(}code for a\textup{)} submeasure, then the class \[\FID(\mu) = \set{X \subseteq \N : \mu(X) < \infty}\] is a free \(F_\sigma\)-ideal.
\end{proposition}

\begin{proof}
  Conditions~\eqref{ID:mono} and~\eqref{ID:free} of Definition~\ref{D:ID} are clear. To verify~\eqref{ID:union}, suppose that \(X_0,\dots,X_k \in \FID(\mu).\) By \Bnd{\Sigma^0_2}, we can find an \(m \in \N\) such that \[\forall n\,\forall i \leq k\,{\mu(X_i \cap \set{0,\dots,n-1}) \leq m}.\] It follows from the subadditivity of \(\mu\) that \[\forall n\,\mu(X \cap \set{0,\dots,n-1}) \leq m(k+1),\] where \(X = X_0 \cup\cdots\cup X_k.\)
\end{proof}

It turns out that all free \(F_\sigma\)-ideals admit representations of this form. This was shown by K.~Mazur~\cite{Mazur} (assuming \(\mathsf{ZFC}\)); we show that the result goes through assuming only \WKL\ + \Bnd{\Sigma^0_2}.

\begin{theorem}[\WKL\ + \Bnd{\Sigma^0_2}]\label{T:Mazur}
  For every free \(F_\sigma\)-ideal \(\mathcal{J}\) there is an integer-valued submeasure \(\mu\) such that \(\mathcal{J} = \FID(\mu).\)
\end{theorem}

\noindent
For our purposes, we will need a slightly more general result which follows from Theorem~\ref{T:Mazur} but whose direct proof is essentially the same.

\begin{proposition}[\WKL\ +  \Bnd{\Sigma^0_2}]\label{P:Submeasure}
  Let \(\seq{T_i}_{i=0}^\infty\) be a sequence of binary trees \textup{(}i.e., codes for closed classes\textup{)}. There is an integer-valued submeasure \(\mu:\fin{\N}\to\N\) such that \(\FID(\mu)\) is the smallest free ideal that contains the \(F_\sigma\) class \(\bigcup_{i=0}^\infty [T_i].\) 
\end{proposition}

\begin{proof}
  Let \(C_0 = \set{\varnothing}\) and, for each \(i \geq 1,\) let \(C_i\) be the set of all \(x \in \fin{\N}\) such that either \(\max(x) < i,\) or \(x \subseteq \tau^{-1}(1)\) for some \(\tau \in \bigcup_{j<i} T_j.\) Note that if there is any such \(\tau\) then there is one with \(|\tau| = \max(x)+1,\) so this is really a finite search. Define \[[C_i] = \set{X \subseteq \N : \forall n\,{X \cap \set{0,\dots,n-1} \in C_i}}.\] Note that \([C_i]\) is a monotone closed subclass of \(\pow{\N}\) and that \(\set{0,\dots,i-1} \in [C_i],\) for every \(i.\)

  For \(x \in \fin{N},\) define \(\theta(x)\) to be the first \(i\) such that \(x \in C_i.\) Since we know that \(x \in C_{\max(x)+1},\) this is again a finite search. Now define \(\mu(\varnothing) = 0\) and \[\mu(x) = \min\set{\theta(z_1) + \cdots + \theta(z_k) : \mbox{$z_1,\dots,z_k$ a partition of $x$}}.\] Since there are only finitely many partitions of the finite set \(x,\) this is again a finite search. It is easy to check that \(\mu\) is (a code for) an integer-valued submeasure.

  To show that \(\FID(\mu)\) is the smallest free ideal that contains the \(F_\sigma\)-class \(\bigcup_{i=1}^\infty [T_i],\) we need three facts.

  \begin{lemma}[\RCA]\label{L:Submeasure:1}
    If \(X \in [C_i]\) then \(\mu(X) \leq i.\)
  \end{lemma}

  \begin{proof}
    For each \(n,\) we have \[\mu(X \cap \set{0,\dots,n-1}) \leq \theta(X \cap \set{0,\dots,n-1}) \leq i. \qedhere\]
  \end{proof}

  \begin{lemma}[\WKL]\label{L:Submeasure:2}
    If \(\mu(X) \leq i\) then there are \(X_1,\dots,X_i \in [C_i]\) such that \(X = X_1 \cup \cdots \cup X_i.\)
  \end{lemma}

  \begin{proof}
    Let \(S\) be the tree of all \(\sigma \in \set{0,\dots,i}^{<\infty}\) such that \(\dom(\sigma) - X = \sigma^{-1}(0)\) and \(\sigma^{-1}(1),\dots,\sigma^{-1}(i) \in C_i.\) It suffices to show that \(S\) is infinite.

    Given \(n,\) we know that \(\mu(X \cap \set{0,\dots,n-1}) \leq i.\) By definition of \(\mu,\) this means that there is a partition \(X \cap \set{0,\dots,n-1} = z_1 \cup \cdots \cup z_k\) such that \(\theta(z_1) + \cdots + \theta(z_k) \leq i.\) It follows that \(k \leq i\) and that \(z_1,\dots,z_k \in C_i.\) This immediately shows that \(S\) has an element with length \(n.\)
  \end{proof}

  Together, these two facts show that \(\FID(\mu)\) is the smallest free ideal that contains the \(F_\sigma\)-class \(\bigcup_{i=0}^\infty [C_i].\) The last fact relates this to the \(F_\sigma\)-class \(\bigcup_{i=0}^\infty [T_i].\)

  \begin{lemma}[\WKL]\label{L:Submeasure:3}
    If \(X \subseteq \N\) is infinite, then \(X \in [C_i]\) iff there are a \(j < i\) and a \(Y \in T_j\) such that \(X \subseteq Y.\)
  \end{lemma}

  \begin{proof}
    We only prove the forward implication since the converse is clear. Fix an \(i\) and let \(T = \bigcup_{j<i} T_j.\)

    Suppose \(X \in [C_i]\) is infinite and let \[T_X = \set{\tau \in T : X \cap \dom(\tau) \subseteq \tau^{-1}(1)}.\] By definition of \(C_i,\) this is an infinite subtree of \(T.\) If \(Y \in [T_X]\) then \(X \subseteq Y,\) so it suffices to show that \(Y \in [T_j]\) for some \(j < i.\) For each \(n\), let \(j_n\) be the first \(j < i\) such that \(\chi_Y \res n \in T_j.\) Since the trees \(T_j\) are downward closed, the sequence \(\seq{j_n}_{n=0}^\infty\) is nondecreasing. So there is a \(j < k\) such that \(j_n = j\) for all sufficiently large \(n\) and then \(Y \in T_j,\) as required.
  \end{proof}
  
  \noindent
  Note that we would need \Ind{\Sigma^0_2} to know exactly which of the sets \(X_1,\dots,X_i\) of Lemma~\ref{L:Submeasure:2} are infinite. However, this doesn't matter since any free ideal that contains \(\bigcup_{i=0}^\infty [T_i]\) must also contain all of the sets \(X_1,\dots,X_i.\) Therefore, \(\FID(\mu)\) is indeed the smallest free ideal that contains the \(F_\sigma\)-class \(\bigcup_{i=0}^\infty [T_i].\)
\end{proof}

Before we finish this discussion of submeasures and their associated ideals, note that the class of submeasures forms a lattice under the pointwise ordering. The joins and meets of this lattice are the following.

\begin{definition}\label{D:Lattice}
  Let \(\mu\) and \(\nu\) be two \textup(codes for\textup) submeasures on \(\N.\)
  \begin{itemize}
  \item \(\mu\vee\nu\) is the submeasure defined by \((\mu\vee\nu)(x) = \max(\mu(x),\nu(x))\) for all \(x \in \fin{\N}.\)
  \item \(\mu\wedge\nu\) is the submeasure defined by \[(\mu\wedge\nu)(x) = \min\set{\mu(y) + \nu(z): x = y \cup  z}\] for all \(x \in \fin{\N}.\)
  \end{itemize}
\end{definition}

\noindent
The meet operation will be very useful since it gives a way of computing the \(F_\sigma\)-ideal generated by two \(F_\sigma\)-ideals.

\begin{proposition}[\WKL\ + \Bnd{\Sigma^0_2}]\label{P:Meet}
  Let \(\mu\) and \(\nu\) be two submeasures on \(\N.\) Then \(\FID(\mu\wedge\nu)\) is the \(F_\sigma\)-ideal generated by \(\FID(\mu) \cup \FID(\nu).\) In other words, \((\mu\wedge\nu)(A) < \infty\) if and only if \(A = B \cup C\) where \(\mu(B) < \infty\) and \(\nu(C) < \infty.\)
\end{proposition}

\noindent
The proof of this fact is left to the reader. The fact that \(\FID(\mu\vee\nu) = \FID(\mu)\cap\FID(\nu)\) is also easy to check but we will have little use for it.

Finally, note that countable joins and meets
\[\bigvee_{n=0}^\infty \mu_n = \sup_{n \in \N} \mu_0\vee\cdots\vee\mu_n \quad\mbox{and}\quad \bigwedge_{n=0}^\infty \mu_n = \inf_{n \in \N} \mu_0\wedge\cdots\wedge\mu_n\] are always welldefined in \ACA, but they do not always correspond exactly to the expected operations on \(F_\sigma\)-ideals. The following fact (and variants) will be used regularly to unite a countable family of \(F_\sigma\)-ideals.

\begin{proposition}[\WKL\ + \Bnd{\Sigma^0_2}]\label{P:InfMeet}
  Let \(\seq{\mu_n}_{n=0}^\infty\) be a sequence of submeasures on \(\N.\) Then \(\FID(\mu)\) is the \(F_\sigma\)-ideal generated by \(\bigcup_{n=0}^\infty \FID(\mu_n),\) where \[\mu = \bigwedge_{n=0}^\infty \mu_n \vee n.\]
\end{proposition}

\noindent
Here and henceforth, we use the constant \(n\) as a convenient abbreviation for the submeasure that assigns measure \(n\) to every nonempty set.

\section{\(F_\sigma\)-Mathias Forcing}\label{S:Forcing}

Having discussed how to handle free \(F_\sigma\)-ideals in second-order arithmetic, we are now ready to describe \(F_\sigma\)-Mathias forcing. Like traditional Mathias forcing, conditions contain a finite part \(a\) and an infinite part \(A \supseteq a.\) These represent the commitment that the generic \(G\) will satisfy \(a \subseteq G \subseteq A.\) In addition, \(F_\sigma\)-Mathias conditions contain a third component, a submeasure \(\mu\) such that \(\mu(A) = \infty.\) This last part represents the commitment that the generic \(G\) will be \(\mu\)-infinite.

\begin{definition}\label{D:FsigmaMathiasFocing}
  \(F_\sigma\)-Mathias forcing is defined as follows.
  \begin{itemize}
  \item Conditions are triples \((a,A,\mu)\) where \(a \in \fin{\N},\) \(a \subseteq A \subseteq \N,\) and \(\mu\) is a (code for a lower semicontinuous) submeasure such that \(\mu(A) = \infty.\) 
  \item The ordering is given by \((b,B,\nu) \leq (a,A,\mu)\) iff \(a \subseteq b \subseteq A,\) \(B \subseteq A,\) and \(\mu \leq \nu.\)
  \end{itemize}
\end{definition}

\noindent
Each condition \((a,A,\mu)\) represents the commitment that the generic real belongs to the \(G_\delta\) class \[\nbhd{a,A,\mu} = \set{B \subseteq \N : a \subseteq B \subseteq A \land \mu(B) = \infty}.\] Together, these classes \(\nbhd{a,A,\mu}\) form a basis for a topology on \(\pow{\N}.\) This topology is finer than the Ellentuck topology~\cite{Ellentuck}, which arises in the same manner for traditional Mathias forcing. We will refer to this finer topology as the \emph{Daguenet topology}, in honor of Maryvonne Daguenet who first studied the corresponding topology on \(\beta\N\)~\cite{Daguenet}.

Given a condition \((a,A,\mu)\) we define \[\tree{a,A} = \set{ \tau \in 2^{<\infty} : a \cap \dom(\tau) \subseteq \tau^{-1}(1) \subseteq A }.\] Note that \(\tree{a,A}\) is a tree and that \(\nbhd{a,A,\mu}\) is a dense \(G_\delta\) subclass of the closed class \([\tree{a,A}].\)

Like traditional Mathias forcing, \(F_\sigma\)-Mathias forcing satisfies Baumgartner's Axiom~A~\cite{Baumgartner}. For $s \in \N$, define \((b,B,\nu) \leq_s (a,A,\mu)\) if \((b,B,\nu) \leq (a,A,\mu),\) \(\nu \wedge s = \mu \wedge s,\) and \(\nu(b) \geq s.\) A sequence \(\seq{a_s,A_s,\mu_s}_{s=0}^\infty\) of conditions such that \((a_{s+1},A_{s+1},\mu_{s+1}) \leq_s (a_s,A_s,\mu_s)\) for each \(s\) is called a \emph{fusion sequence}. Given such a sequence, the infimum \(\mu = \bigwedge_{s=0}^\infty \mu_s\) is a well-defined submeasure and \(\mu(A) = \infty,\) where \(A = \bigcup_{s=0}^\infty a_s.\) Moreover, each \((a_s,A,\mu)\) is a condition such that \((a_s,A,\mu) \leq_s (a_s,A_s,\mu_s).\)

This Fusion Lemma differs slightly from Baumgartner's condition~(3). It is not difficult to modify our definition of the partial order ${\leq_s}$ to satisfy Baumgartner's definition, but we chose a version that works better in our more restricted context. Baumgartner's final condition~(4) doesn't make much sense in our context since it involves quantification over third-order objects; condition~(4$'$) is more meaningful but still awkward to work with. We haven't found a useful variant of this condition in our context, but Lemma~\ref{L:StarFusion} captures the useful part of this condition.

\subsection{Forcing Language}

We will now develop the basic machinery necessary to define the internal forcing language. The base level of this are the forcing names, which are the terms of the forcing language.

\begin{definition}\label{D:Names}
  A \emph{partial \(k\)-ary name} is a \(\Sigma^0_1\) set \(F \subseteq 2^{<\infty}\times\N^{k+1}\) such that:
  \begin{itemize}
  \item If \((\tau,\bar{x},y) \in F\) and \(\tau \subseteq \tau'\) then \((\tau',\bar{x},y) \in F.\)
  \item If \((\tau,\bar{x},y), (\tau,\bar{x},y') \in F\) then \(y = y'.\)
  \end{itemize}
  The \emph{domain of \(F\)} is the \(G_\delta\) class \[\dom(F) = \set{ X \subseteq \N : \forall\bar{x}\,\exists \tau,y\,(\tau \subseteq \chi_X \land (\tau,\bar{x},y) \in F) }.\] Given \(X \in \dom(F),\) the \emph{evaluation} \[F^X(\bar{x}) = y\ \IFF\ \exists \tau\,(\tau \subseteq \chi_X \land (\tau,\bar{x},y) \in F)\] is a total \(k\)-ary function.
\end{definition}

The reader will recognize these names as Turing functionals (relative to ground model oracles). The first coordinate should then be thought as finite amount of information from the generic real. When \(G\) is the \(F_\sigma\)-Mathias generic real and \(G \in \dom(F)\), the evaluation \(F^G\) is the intended interpretation of the name \(F\) in the generic extension. The basic projections, constants, and indeed all ground model functions \(F\) have \emph{canonical names} \(\check{F}\) defined by \[(\tau,\bar{x},y) \in \check{F}\ \IFF\ y = F(\bar{x}),\] which invariably evaluate to \(F.\)

In a typical language, the basic terms are composed to form the class of all terms. This is not so for the forcing language since composition and other operations can be done directly at the semantic level. If \(F\) is a partial \(\ell\)-ary name and \(F_1,\dots,F_\ell\) and are partial \(k\)-ary names then the superposition \(H = F (F_1,\dots,F_\ell)\) is defined by \[(\tau,\bar{x},z) \in H\ \IFF\ \exists \bar{y}\,((\tau,\bar{x},y_1) \in F_1 \land\cdots\land (\tau,\bar{x},y_\ell) \in F_\ell \land (\tau,\bar{y},z) \in F).\] This is a partial \(k\)-ary name and \[\dom(H) \supseteq \dom(F_1) \cap \cdots \cap \dom(F_\ell) \cap \dom(F).\] Primitive recursion can be handled in a similar way. Given partial a \((k-1)\)-ary name \(F_0\) and a \((k+1)\)-ary name \(F\), we the \(k\)-ary name \(H\) is defined by \((\tau,\bar{x},y,z) \in H\) iff there is a finite sequence \(\seq{z_0,\dots,z_y}\) with \(z = z_y\) such that \((\tau,\bar{x},z_0) \in F_0\) and \((\tau,\bar{x},i,z_i,z_{i+1}) \in F\) for every \(i < y\). This is a partial \(k\)-ary name and \[\dom(H) \supseteq \dom(F_0) \cap \dom(F).\] 
(Note that \(\Sigma^0_1\)-induction is necessary to establish this last fact, whereas none is needed to establish the corresponding fact for superposition.)
Other recursive operations will be handled later in Corollary~\ref{C:Sigma1Uniformization}. Meanwhile, we can proceed by defining the formulas of the forcing language.

\begin{definition}\label{D:Formulas}
  The \emph{formulas} of the forcing language are defined in the usual manner as the smallest family which is closed under the following formation rules.
  \begin{itemize}
  \item If \(F\) is a partial \(k\)-ary name, \( F'\) is a partial \(k'\)-ary name, and \(\bar{v} = v_1,\dots,v_k\), \(\bar{v}' = v_1',\dots,v_{k'}'\) are variable symbols then \(F(\bar{v}) = F'(\bar{v}')\) is a formula.
  \item If \(\phi\) is a formula then so is \(\lnot\phi.\)
  \item If \(\phi\) and \(\psi\) are formulas then so is \(\phi\land\psi.\)
  \item If \(\phi\) is a formula and \(x\) is a variable symbol, then \(\forall x\,\psi\) is also a formula.
  \end{itemize}
  Free and bound variables are defined in the usual manner. The \emph{sentences} of the forcing language are formulas without free variables.
\end{definition}

\noindent
Although not present in the formal language, we will often use \(\lor,\) \(\lthen,\) \(\liff\) and \(\exists\) as abbreviations:
\[\begin{array}{c@{\qquad}c}
  \phi \lor \psi \equiv \lnot(\lnot\phi\land\lnot\psi), &
  \phi \lthen \psi \equiv \lnot(\phi\land\lnot\psi), \\
  \phi \liff \psi \equiv (\phi\lthen\psi)\land(\psi\lthen\phi), &
  \exists v\,\phi \equiv \lnot\forall v\,\lnot\phi.
\end{array}\]
The language does not include second-order variable symbols. We will not have any need for second-order quantification, so it would be unnecessary tedium to introduce such variables.\footnote{We will use one second order variable in Section~\ref{S:Witnessing}, which will be handled in an \emph{ad hoc} manner. The interested reader can use this model to extend the language.}

Names are intended to represent functions in the generic extension. Not all names are equally meaningful in this way. The canonical names all have perfectly reasonable meaning, but the empty name has no reasonable interpretation. Locality is the notion that distinguishes meaningful names from pathological ones.

\begin{definition}\label{D:Locality}
  Let \((a,A,\mu)\) be an \(F_\sigma\)-Mathias condition.
  \begin{itemize}
  \item We say that \(F\) is an \emph{\((a,A,\mu)\)-local name} if \(\nbhd{b,B,\nu} \cap \dom(F) \neq \varnothing\) for every extension \((b,B,\nu) \leq (a,A,\mu).\)
  \item We say that \(\phi\) is an \emph{\((a,A,\mu)\)-local formula} of the forcing language if every name that occurs in \(\phi\) is \((a,A,\mu)\)-local.
  \end{itemize}
\end{definition}

\noindent
Note that \(F\) is \((a,A,\mu)\)-local precisely when \(\dom(F)\cap\nbhd{a,A,\mu}\) is dense in \(\nbhd{a,A,\mu}\) with respect to the Daguenet topology. Thus canonical names are always \((a,A,\mu)\)-local while the empty name is never \((a,A,\mu)\)-local.

\subsection{The Forcing Relation}

We are now ready to define the forcing relation. The definition for atomic sentences will be motivated when we discuss the forcing extension. The remaining cases follow the classical definition of forcing.

\begin{definition}\label{D:ForcingRelation}
  The forcing relation \((a,A,\mu) \forces \theta\) is defined by induction on the complexity of the \((a,A,\mu)\)-local sentence \(\theta\) as follows. Assume all names that occur in sentences below are \((a,A,\mu)\)-local.
  \begin{itemize}
  \item \((a,A,\mu) \forces F = F'\) iff, for all \(\tau \in \tree{a,A}\) and \(y, y' \in \N,\) if \((\tau,y) \in F\) and \((\tau,y') \in F'\) then \(y = y'.\)
  \item \((a,A,\mu) \forces \phi\land\psi\) iff \((a,A,\mu) \forces \phi\) and \((a,A,\mu) \forces \psi.\)
  \item \((a,A,\mu) \forces \forall v\,\phi(v)\) iff \((a,A,\mu) \forces \phi(x),\) for all \(x \in \N.\)
  \item \((a,A,\mu) \forces \lnot\phi\) iff there is no \((b,B,\nu) \leq (a,A,\mu)\) such that \((b,B,\nu) \forces \phi.\)
  \end{itemize}
\end{definition}

\noindent
The meaning of the forcing relation for the abbreviations defined above can be computed as usual.

\begin{proposition}[\RCA]\label{P:ForcingRelation}
  Assume all names that occur in the sentences below are \((a,A,\mu)\)-local.
  \begin{itemize}
  \item \((a,A,\mu) \forces F \neq F'\) iff, for all \(\tau \in \tree{a,A}\) and \(y,y' \in \N,\) if \((\tau,y) \in F\) and \((\tau,y') \in F'\) then \(y \neq y'.\)
  \item \((a,A,\mu) \forces \phi\lor\psi\) iff for every \((b,B,\nu) \leq (a,A,\mu)\) there is a \((c,C,\kappa) \leq (b,B,\nu)\) such that either \((c,C,\kappa) \forces \phi\) or \((c,C,\kappa) \forces \psi.\)
  \item \((a,A,\mu) \forces \phi\lthen\psi\) iff for every \((b,B,\nu) \leq (a,A,\mu)\) such that \((b,B,\nu) \forces \phi,\) there is a \((c,C,\kappa) \leq (b,B,\nu)\) such that \((c,C,\kappa) \forces \psi.\)
  \item \((a,A,\mu) \forces \exists v\,\phi(v)\) iff for every \((b,B,\nu) \leq (a,A,\mu)\) there are a \((c,C,\kappa) \leq (b,B,\nu)\) and a \(x \in \N\) such that \((c,C,\kappa) \forces \phi(x).\)
  \item \((a,A,\mu) \forces \lnot\lnot\phi\) iff \((a,A,\mu) \forces \phi.\)
 \end{itemize}
\end{proposition}

\noindent
The verifications of the above are straightforward.

According to the definition given, the complexity of the forcing relation is highly complex even for simple sentences. For example, the complexity of \((a,A,\mu) \forces \exists w\,{F(w) = 0}\) is technically \(\Pi^1_2.\) We will spend much time reducing this complexity so that we can ``comprehend'' statements in the forcing extension from the ground model. First and foremost, we need to understand the \(\Pi^0_1\) forcing relation. In Proposition~\ref{P:Pi1Forcing}, we will show that this relation is itself \(\Pi^0_1,\) but first we handle the bounded forcing relation.

\begin{proposition}[\RCA]\label{P:BoundedForcing}
  For every bounded formula \(\phi(\bar{v})\) of the forcing language, there is a partial name \(T_\phi(\bar{v})\) such that, for every condition \((a,A,\mu),\) if \(\phi(\bar{v})\) is \((a,A,\mu)\)-local then so is \(T_\phi(\bar{v})\) and \[(a,A,\mu) \forces \forall\bar{v}\,(\phi(\bar{v}) \liff T_\phi(\bar{v}) = 0).\]
\end{proposition}

\begin{proof}
  The partial names \(T_\phi(\bar{v})\) are defined by induction on the complexity of \(\phi\) as follows.
  \begin{itemize}
  \item If \(\phi(\bar{v}) \equiv F(\bar{v}) = F'(\bar{v})\) then \(T_\phi(\bar{v}) = |F(\bar{v}) - F'(\bar{v})|.\)
  \item If \(\phi(\bar{v}) \equiv \lnot\psi(\bar{v})\) then \(T_\phi(\bar{v}) = 1 \sub T_\psi(\bar{v}).\)
  \item If \(\phi(\bar{v}) \equiv \psi(\bar{v}) \land \theta(\bar{v})\) then \(T_\phi(\bar{v}) = T_\psi(\bar{v}) + T_\theta(\bar{v}).\)
  \item If \(\phi(\bar{v}) \equiv \forall w \leq F(\bar{v})\,\psi(\bar{v},w)\) then \(T_\phi(\bar{v}) = \sum_{w \leq F(\bar{v})} T_\psi(\bar{v},w).\)
  \end{itemize}
  Verifications are straightforward.
\end{proof}

\noindent
We can now show that the \(\Pi^0_1\) forcing relation is \(\Pi^0_1.\) 

\begin{proposition}[\RCA]\label{P:Pi1Forcing}
  If \(\phi(\bar{v})\) is a \(\Pi^0_1\) formula of the forcing language, then there is a \(\Pi^0_1\) formula \(\widehat{\phi}(a,A;\bar{v})\) such that, for every condition \((a,A,\mu)\) and all \(\bar{x} \in \N,\) if \(\phi(\bar{x})\) is \((a,A,\mu)\)-local, then \[(a,A,\mu) \forces \phi(\bar{x})\ \IFF\ \widehat{\phi}(a,A;\bar{x}).\]
\end{proposition}

\begin{proof}
  Write \(\phi(\bar{v})\) in the form \(\forall w\,(T(\bar{v},w) = 0),\) where \(T(\bar{v},w)\) is a partial name as in Proposition~\ref{P:BoundedForcing} for \(\phi_0(\bar{x},y).\) Then, by Definition~\ref{D:ForcingRelation}, \((a,A,\mu) \forces \phi(\bar{x})\) iff \[\forall \tau,y,z\,(\tau \in \tree{a,A} \land (\tau,\bar{x},y,z) \in T \lthen z = 0),\] which is the desired \(\Pi^0_1\) formula \(\widehat{\phi}(a,A;\bar{x}).\)
\end{proof}

\noindent
Note that the question of locality was conveniently factored out in Proposition~\ref{P:Pi1Forcing}. This is necessary since locality is generally \(\Pi^1_2.\) However, the the following lemma can be used to reduce the complexity of locality.

\begin{proposition}[\WKL\ + \Bnd{\Sigma^0_2}]\label{P:FunTotal}
  For every partial \(k\)-ary name \(F\) there is a sequence of submeasures \(\seq{\vartheta_{a}:a\in\fin{\N}}\) such that the following hold for every condition \((a,A,\mu).\)
  \begin{itemize}
  \item If \((\mu\wedge\vartheta_b)(A) < \infty\) for some \(a \subseteq b \subseteq A\) then there is an extension \((b,B,\mu) \leq (a,A,\mu)\) such that \(\nbhd{b,B,\mu} \cap \dom(F) = \varnothing.\) 
  \item If \((\mu\wedge\vartheta_b)(A) = \infty\) for every \(a \subseteq b \subseteq A\) then \(\nbhd{a,A,\mu} \cap \dom(F) \neq \varnothing.\) 
  \end{itemize}
\end{proposition}

\begin{proof}
  Use Proposition~\ref{P:Submeasure} to find a submeasure \(\vartheta_b\) such that \(\FID(\vartheta_b)\) is the free ideal generated by the \(F_\sigma\) class \[\mathcal{B}_b = \set{B \subseteq \N : \exists\bar{x}\,\forall \tau,y\,(\tau \in \tree{b,B \cup b} \lthen (\tau,\bar{x},y) \notin F)}.\] Since the procedure of Proposition~\ref{P:Submeasure} is uniform, the sequence \(\seq{\vartheta_b:b \in \fin{\N}}\) can be computed effectively.

  Suppose that \((a,A,\mu)\) is a condition such that \((\mu\wedge\vartheta_b)(A) < \infty\) for some \(b \supseteq a.\) Then there is a decomposition \(A = B_0 \cup B_1 \cup \cdots\cup B_k\) such that \(\mu(B_0) < \infty\) and \(B_1 \cup a,\dots,B_k \cup a \in \mathcal{B}_b.\) Since \(\mu(A) = \infty\) there must be a \(B_i\) with \(\mu(B_i) = \infty.\) Let \(B = B_i \cup a\) and find \(a \subseteq b \subseteq B\) and \(\bar{x}\) witnessing that \(B_i \in \mathcal{B}_{F,a}.\)

  Suppose that \((a,A,\mu)\) is a condition such that \((\mu\wedge\vartheta_b)(A) = \infty\) for all \(b \supseteq a.\) Define a sequence \(\seq{a_n,a'_n}_{n=0}^\infty\) of pairs of finite subsets of \(A\) by recursion. Let \(\seq{\bar{x}_n}_{n=0}^\infty\) enumerate \(\N^k.\) Initially, set \(a_0 = a\) and \(a'_0 = \varnothing.\) At stage \(n,\) first find \(\tau_n \in U(a_n,A-a'_n)\) and \(y_n \in \N\) such that \((\tau_n,\bar{x}_n,y_n) \in F\) and \(|\tau_n| \geq n.\) This is possible since \(\vartheta_{a_n}(A-a'_n) = \infty,\) which guarantees that \(A-a'_n \notin \mathcal{B}_{a_n}.\) Then set \(a'_{n+1} = a'_n \cup (\tau_n^{-1}(0) \cap A)\) and then pick \(a_{n+1}\) so that \(a_n \cup \tau_n^{-1}(1) \subseteq a_{n+1} \subseteq A-a'_{n+1}\) and \(\mu(a_{n+1}) \geq n+1.\) At the end of the construction, we have a set \(\bigcup_{n=0}^\infty a_n \in [a,A,\mu] \cap \dom(F).\)
\end{proof}

\noindent
Note that in the second alternative, we can form the submeasure \[\vartheta = \bigwedge_{a \subseteq b \subseteq A} (\vartheta_b \vee |b|)\] and then \((a,A,\mu\wedge\vartheta)\) is an extension of \((a,A,\mu)\) such that \(F\) is \((a,A,\mu\wedge\vartheta)\)-local. 

A similar trick as in Proposition~\ref{P:FunTotal} can be used to control the complexity of the \(\Pi^0_2\) forcing relation.

\begin{proposition}[\WKL\ + \Bnd{\Sigma^0_2}]\label{P:Pi2Decision}
  If \(\phi(w)\) is an \((a,A,\mu)\)-local \(\Pi^0_1\) formula of the forcing language, then there is an integer-valued submeasure \(\varrho\) with the following properties.
  \begin{itemize}
  \item If \((\mu\wedge\varrho)(A) < \infty\) then there are an extension \((b,B,\mu) \leq (a,A,\mu)\) and a \(y \in \N\) such that \((b,B,\mu) \forces \phi(y).\)
  \item If \((\mu\wedge\varrho)(A) = \infty\) then \((a,A,\mu\wedge\varrho) \forces \forall w\,\lnot\phi(w).\)
  \end{itemize}
\end{proposition}

\begin{proof}
  Let \(\widehat{\phi}(a,A;\bar{v},w)\) be as in Proposition~\ref{P:Pi1Forcing}. Using the procedure of Proposition~\ref{P:Submeasure}, compute an integer-valued submeasure \(\varrho\) so that \(\FID(\varrho)\) is the smallest free ideal that contains the monotone \(F_\sigma\)-class \[\mathcal{F} = \set{B \subseteq A : \exists b, y\,(a \subseteq b \subseteq A \land \widehat{\phi}(b,B \cup b;y)}.\] Thus, in particular, if \((b,B,\nu) \leq (a,A,\mu)\) and \((b,B,\nu) \forces \phi(y)\) for some \(y \in \N,\) then \(\varrho(B) < \infty.\) The second statement follows immediately from this observation.

  For the first statement, note that if \((\mu\wedge\varrho)(A) < \infty\) then there are \(B_0 \in \FID(\mu)\) and \(B_1,\dots,B_k \in \mathcal{F}\) (\(k \in \N\)) such that \[A = B_0 \cup B_1 \cup \cdots \cup B_k.\] Since \(\mu(A) = \infty,\) there must be an \(i \in \set{1,\dots,k}\) such that \(\mu(B_i) = \infty.\) Then, by definition of \(\mathcal{F}\), we can find \(b \subseteq A\) and \(y \in \N\) such that if \(B = b \cup B_i\) then \((b,B,\mu) \leq (a,A,\mu)\) is a condition with \((b,B ,\mu) \forces \phi(y),\) as required.
\end{proof}

\section{Witnessing Theorems}\label{S:Witnessing}

The key to reducing the complexity of the forcing relation is to eliminate existential quantifiers by introducing witnessing terms. We will do this by producing names for Skolem terms (or dually Herbrand terms).

\subsection{Skolemization}

\begin{definition}\label{D:Skolem}
  The \emph{Skolemization} \(\theta_S(W;\bar{v})\) and the dual \emph{Herbrandization} \(\theta_H(W;\bar{v})\) of a formula \(\theta(\bar{v})\) of the forcing language are defined by induction on the complexity of \(\theta\) as follows. In all cases, the formal parameter \(W\) is a place holder for a unary name. We will use \(\lambda\)-notation to distinguish the parameter for \(W\) from true variable symbols of the ambient formula.
  \begin{itemize}
  \item If \(\theta(\bar{v})\) is atomic, then \[\theta_S(W;\bar{v}) \equiv \theta(\bar{v}) \equiv \theta_H(W;\bar{v}).\]
  \item If \(\theta(\bar{v}) \equiv \lnot\phi(\bar{v}),\) then \[\theta_S(W;\bar{v}) \equiv \lnot\phi_H(W;\bar{v}), \quad \theta_H(W;\bar{v}) \equiv \lnot \phi_S(W;\bar{v}).\]
  \item If \(\theta(\bar{v}) \equiv \phi(\bar{v}) \land \psi(\bar{v}),\) then \[\theta_S(W;\bar{v}) \equiv \phi_S(\lambda t\,W(2t);\bar{v}) \land \psi_S(\lambda t\,W(2t+1);\bar{v}),\] and \[\theta_H(W;\bar{v}) \equiv \phi_H(W;\bar{v}) \land \psi_H(W;\bar{v}).\]
  \item If \(\theta(\bar{v}) \equiv \forall w\,\phi(\bar{v},w)\) then \[\theta_S(W;\bar{v}) \equiv \forall w\,\phi_S(\lambda t\,W(\pair{w}{t});\bar{v},w),\] and \[\theta_H(W;\bar{v}) \equiv \phi_H(\lambda t\,W(t+1);\bar{v},W(0)).\]
  \end{itemize}
\end{definition}

\noindent
The point of the above definitions is that the Skolemization \(\phi_S(W;\bar{v})\) is essentially a \(\Pi^0_1\)-formula, up to the usual syntactic transformations; dually, the Herbrandization \(\phi_H(W;\bar{v})\) is essentially a \(\Sigma^0_1\)-formula.

\begin{proposition}[\RCA]\label{P:Pi1Skolem}
  If \(\phi(\bar{v})\) is a formula of the forcing language, then there is a \(\Pi^0_1\) formula \(\widetilde{\phi}(a,A,W;\bar{v})\) such that, for every condition \((a,A,\mu),\) every partial name \(W(\bar{v},t),\) and all \(\bar{x} \in \N,\) if \(\phi(\bar{x})\) and \(\lambda t\,W(\bar{x},t)\) are \((a,A,\mu)\)-local, then \[(a,A,\mu) \forces \phi_S(\lambda t\,W(\bar{x},t);\bar{x})\ \IFF\ \widetilde{\phi}(a,A,W;\bar{x}).\]
\end{proposition}

\begin{proof}
  Chasing through the cases of Definition~\ref{D:Skolem}, we see that the only quantifiers that occur in \(\phi_S(\lambda t\,W(\bar{x},t);\bar{v})\) are universal and that they all occur positively (i.e.\ within the scope of an even number of negations). Let \(\psi(\lambda t\,W(\bar{x},t);u,\bar{v})\) be the the bounded formula obtained from \(\phi_S(\lambda t\,W(\bar{x},t);\bar{v})\) by bounding all universal quantifiers with the fresh variable \(u.\) If \(\phi(\bar{x})\) and \(\lambda t\,W(\bar{x},t)\) are \((a,A,\mu)\)-local then so are \(\phi_S(\lambda t\,W(\bar{x},t);\bar{x})\) and \(\psi(\lambda t\,W(\bar{x},t);u,\bar{x})\) and \[(a,A,\mu) \forces \phi_S(\lambda t\,W(\bar{x},t);\bar{x})\ \IFF\ (a,A,\mu) \forces \forall u\,\psi(\lambda t\,W(\bar{x},t);u,\bar{x}).\] The result follows by defining \(\widetilde{\phi}(a,A,W;\bar{v})\) to be the \(\Pi^0_1\)-formula associated to \(\forall u\,\psi(\lambda t\,W(\bar{v},t);u,\bar{v})\) as in Proposition~\ref{P:Pi1Forcing}.
\end{proof}

\noindent
In view of this, the complexity of the forcing relation can be reduced by finding appropriate Skolem names. We will first do this for \(\Pi^0_2\) sentences in \RCA, then we will extend our results to sentences of higher arithmetical complexity in \ACA. First, let us introduce some convenient terminology for this task.

\begin{definition}
  The \emph{Skolemized forcing relation} \({\Sforces}\) and the \emph{Herbrandized forcing relation} \({\Hforces}\) are defined as follows. Let \(\theta\) be an \((a,A,\mu)\)-local sentence of the forcing language.
  \begin{itemize}
  \item \((a,A,\mu) \Sforces \theta\) holds iff \((a,A,\mu) \forces \theta_S(W)\) for some \((a,A,\mu)\)-local unary name~\(W.\)
  \item \((a,A,\mu) \Hforces \theta\) holds iff \((a,A,\mu) \forces \theta_H(W)\) for every \((a,A,\mu)\)-local unary name~\(W.\)
  \end{itemize}
\end{definition}

\noindent
These forcing-like relations are related to the forcing relation as follows.

\begin{proposition}[\RCA]
  If \(\theta\) is a \((a,A,\mu)\)-local sentence of the forcing language, then \[(a,A,\mu) \Sforces \theta \THEN (a,A,\mu) \forces \theta \THEN (a,A,\mu) \Hforces \theta.\]
\end{proposition}

\noindent
The converses of the above implications do not generally hold. A result that establishes a partial converse to one of the above is called a \emph{witnessing theorem}. In the remainder of this section, we will prove some standard witnessing theorems in \RCA\ and \ACA.

\subsection{Witnessing in \RCA}\label{S:Witnessing:RCA}

Our first main result of this section is that witnessing \(\Pi^0_2\) sentences is automatic in \RCA.

\begin{theorem}[\RCA]\label{T:Pi2Witnessing}
  If \(\theta\) is an \((a,A,\mu)\)-local \(\Pi^0_2\) sentence of the forcing language, then \((a,A,\mu) \forces \theta \IFF (a,A,\mu) \Sforces \theta.\)
\end{theorem}

\noindent
As an immediate consequence, we have a slightly weaker form of witnessing for \(\Sigma^0_3\) sentences which is obtained by first extending the condition to witness the outer existential quantifier.

\begin{corollary}[\RCA]\label{C:Sigma3Witnessing}
  If \(\theta\) is an \((a,A,\mu)\)-local \(\Sigma^0_3\) sentence of the forcing language such that \((a,A,\mu) \forces \theta,\) then there is an extension \((b,B,\nu) \leq (a,A,\mu)\) such that \((b,B,\mu) \Sforces \theta.\) 
\end{corollary}

\noindent
The main use of Theorem~\ref{T:Pi2Witnessing} is that \(\Sigma^0_1\) formulas of the forcing language can be uniformized by names.

\begin{corollary}[\RCA]\label{C:Sigma1Uniformization}
  If \(\phi(\bar{v},w)\) is an \((a,A,\mu)\)-local \(\Sigma^0_1\) formula of the forcing language such that \((a,A,\mu) \forces \forall\bar{v}\,\exists w\,\phi(\bar{v},w),\) then there is an \((a,A,\mu)\)-local name \(F(\bar{v})\) such that \((a,A,\mu) \forces \forall\bar{v}\,\phi(\bar{v},F(\bar{v})).\)
\end{corollary}

\noindent
These results are optimal since \RCA\ cannot prove the existence of Skolem functions for \(\Pi^0_3\) sentences (of the standard language).

In what follows, we will prove Theorem~\ref{T:Pi2Witnessing} incrementally, starting with bounded sentences. The statements we will need are also more complex since we additionally need some uniformity in the witnessing names to bootstrap our way up in the arithmetical hierarchy.

\begin{proposition}[\RCA]\label{P:BoundedWitnessing}
  If \(\theta(\bar{v})\) is a bounded formula of the forcing language  then there are partial names \(W^\theta_S(\bar{v},t)\) and \(W^\theta_H(\bar{v},t)\) such that, for all \(\bar{x} \in \N\), if \(\theta(\bar{x})\) is \((a,A,\mu)\)-local then so are \(\lambda t\,W^\theta_S(\bar{x},t)\) and \(\lambda t\,W^\theta_H(\bar{x},t),\) and \[\begin{aligned}
  (a,A,\mu) \forces \theta(\bar{x})
  &\IFF (a,A,\mu) \forces \theta_S(\lambda t\,W^\theta_S(\bar{x},t);\bar{x}), \\
  &\IFF (a,A,\mu) \forces \theta_H(\lambda t\,W^\theta_H(\bar{x},t);\bar{x}).
  \end{aligned}\]
\end{proposition}

\begin{proof}
  We define \(W^\theta_S(\bar{v},t)\) and \(W^\theta_H(\bar{v},t)\) by induction on the complexity of~\(\theta(\bar{v}).\)
  \begin{itemize}
  \item If \(\theta(\bar{v})\) is atomic, define \(W^\theta_S(\bar{v},t) = 0\) and \(W^\theta_H(\bar{v},t) = 0,\) identically.
  \item If \(\theta(\bar{v}) \equiv \lnot\phi(\bar{v}),\) define \(W^\theta_S(\bar{v},t) = W^\phi_H(\bar{v},t)\) and \(W^\theta_H(\bar{v},t) = W^\phi_S(\bar{v},t).\)
  \item If \(\theta(\bar{v}) \equiv \phi(\bar{v})\land\psi(\bar{v}),\) define
    \[W^\theta_S(\bar{v},2t+i) = \begin{cases} 
      W^\phi_S(\bar{v},t) & \text{when $i = 0$,} \\
      W^\psi_S(\bar{v},t) & \text{when $i = 1$;}
    \end{cases}\]
    and
    \[W^\theta_H(\bar{v},t) = \begin{cases} 
      W^\phi_H(\bar{v},t) & \text{when $T_\phi(\bar{v}) \neq 0$,} \\ 
      W^\psi_H(\bar{v},t) & \text{when $T_\phi(\bar{v}) = 0$,} 
    \end{cases}\] 
    where \(T_\phi(\bar{v})\) is as in Proposition~\ref{P:BoundedForcing}.
  \item If \(\theta(\bar{v}) \equiv \forall w \leq F(\bar{v})\,\phi(\bar{v},w),\) define
    \[W^\theta_S(\bar{v},t) = \begin{cases}
      W^{\phi}_S(\bar{v},\fst{t},\lfloor\snd{t}/2\rfloor) & \text{when $\fst{t} \leq F(\bar{v})$,} \\
      0 & \text{otherwise.}
    \end{cases}\]
    To define \(W^\theta_H,\) first define \[W^\theta_H(\bar{v},0) = \sum_{w \leq F(\bar{v})} U(\bar{v},w),\] where \[U(\bar{v},w) = 1\sub\sum_{u \leq w} T_\phi(\bar{v},u)\] and \(T_\phi(\bar{v},w)\) is as in Proposition~\ref{P:BoundedForcing}. Thus, in plain language, \(W^\theta_H(\bar{v},0)\) is either the first \(w \leq F(\bar{v})\) such that \(\phi(\bar{v},w)\) fails, or \(W^\theta_H(\bar{v},0) = F(\bar{v})+1\) if there is no such \(w.\) Then, define \[W^\theta_H(\bar{v},t+1) = 
    \begin{cases}
      W^\phi_H(\bar{v},W^\theta_H(\bar{v},0),t) & \text{when \(W^\theta_H(\bar{v},0) \leq F(\bar{v})\),} \\
      0 & \text{otherwise.}
    \end{cases}\]
  \end{itemize}
  Verifications are straightforward.
\end{proof}

\noindent
Note that the conclusion of Proposition~\ref{P:BoundedWitnessing} is considerably stronger than that of Theorem~\ref{T:Pi2Witnessing}. This strong form of witnessing extends to \(\Pi^0_1\) sentences, but only for Skolemized forcing and not for Herbrandized forcing.

\begin{proposition}[\RCA]\label{P:Pi1Witnessing}
 If \(\phi(\bar{v})\) is a \(\Pi^0_1\) formula of the forcing language then there is a partial name \(W^\phi_S(\bar{v},t)\) such that, for all \(\bar{x} \in \N,\) if \(\phi(\bar{x})\) is \((a,A,\mu)\)-local then so is \(\lambda t\,W^\phi_S(\bar{x},t),\) and \[(a,A,\mu) \forces \phi(\bar{x}) \IFF (a,A,\mu) \forces \phi_S(\lambda t\,W^\phi_S(\bar{x},t);\bar{x}).\] \end{proposition}

\begin{proof}
  Suppose \(\phi(\bar{v}) \equiv \forall w\,\theta(\bar{v},w),\) where \(\theta(\bar{v},w)\) is bounded. Then define \(W^\phi_S(\bar{v},t) = W^{\theta}_S(\bar{v},\fst{t},\snd{t}),\) where \(W^{\theta}_S(\bar{v},w,t)\) is as in Proposition~\ref{P:BoundedWitnessing}. This \(W^\phi_S(\bar{v},t)\) is as required.
\end{proof}

Theorem~\ref{T:Pi2Witnessing} is an immediate consequence of the following result, which additionally shows that the witnessing names can be chosen uniformly.

\begin{proposition}[\RCA]\label{P:Pi2Witnessing}
  If \(\phi(\bar{v})\) is a \(\Pi^0_2\) formula of the forcing language then there is a partial name \(W^\phi_S(\bar{v},t)\) such that, for all \((a,A,\mu)\) and all \(\bar{x} \in \N,\) if \(\phi(\bar{x})\) is \((a,A,\mu)\)-local and \((a,A,\mu) \forces \phi(\bar{x}),\) then \(\lambda t\,W^\phi_S(\bar{x},t)\) is \((a,A,\mu)\)-local and \((a,A,\mu) \forces \phi_S(\lambda t\,W^\phi_S(\bar{x},t);\bar{x}).\)
\end{proposition}

\begin{proof}
  First note that if \(W^\phi_S(\bar{v},w,t)\) is as required for the formula \(\phi(\bar{v},w),\) then \(W^\psi_S(\bar{v},t) = W^\phi_S(\bar{v},\fst{t},\snd{t})\) is as required for the formula \(\forall w\,\phi(\bar{v},w).\) Therefore, it suffices to handle the case when \(\phi(\bar{v})\) is a \((a,A,\mu)\)-local \(\Sigma^0_1\) formula.

  Suppose first that \(\phi(\bar{v}) \equiv \exists w\,\theta(\bar{v},w),\) where \(\theta(\bar{v},w)\) is bounded. Following the lead of Proposition~\ref{P:BoundedWitnessing}, we would like to define \[W^\phi_S(\bar{v},0) = \sum_{w = 0}^\infty U(\bar{v},w),\] where \[U(\bar{v},w) = 1 \sub \sum_{u \leq w} T_{\lnot\theta}(\bar{v},u)\] and \(T_{\lnot\theta}(\bar{v},w)\) is as in Proposition~\ref{P:BoundedForcing}, and then continue by defining \[W^\phi_S(\bar{v},t+1) = W^\theta_S(\bar{v},W^\phi_S(\bar{v},0),t),\] where \(W^\theta_S(\bar{v},w,t)\) is as in Proposition~\ref{P:BoundedWitnessing}. Indeed, if the above infinite series converges, then \(W^\phi_S(\bar{v},0)\) is the first \(w\) such that \(\theta(\bar{v},w)\) holds and then \(\lambda t\,W_S^\phi(\bar{v},t+1)\) witnesses that fact. In order to sum the above infinite series effectively, we use the fact that \(U(\bar{v},w)\) can only switch from $1$ to $0$ once as $w$ increases. Thus, we can define \[W^\phi_S(\bar{v},0) = y \IFF \exists s\left(y < s \land y = {\textstyle\sum_{w \leq s} U(\bar{v},w)}\right).\] This is a \(\Sigma^0_1\) formula and the remaining values \(W^\phi_S(\bar{v},t+1)\) are defined as above.

  We show that if \((a,A,\mu)\) is such that \(\phi(\bar{x})\) is \((a,A,\mu)\)-local and \((a,A,\mu) \forces \phi(\bar{x})\) then \(\dom(\lambda t\,W^\phi_S(\bar{x},t)) \cap [a,A,\mu] \neq \varnothing.\) To see this, first find \((b,B,\nu) \leq (a,A,\mu)\) and \(w_0 \in \N\) such that \((b,B,\nu) \forces \theta(\bar{x},w_0).\) Then \[\dom(W^\phi_S(\bar{x},0)) \cap [b,B,\nu] \supseteq \dom(\lambda w\,U(\bar{x},w)) \cap [b,B,\nu],\] which implies that \(\dom(\lambda t\,W^\phi_S(\bar{x},t)) \cap [n,B,\nu]\) contains \[\dom(\lambda w\,U(\bar{x},w)) \cap \dom(\lambda t\,W^\theta_S(\bar{x},t)) \cap [b,B,\nu].\] Since \(U(\bar{v},w)\) and \(W^\theta_S(\bar{v},t)\) are \((a,A,\mu)\)-local, it follows that \[\dom(\lambda t\,W^\phi_S(\bar{x},t)) \cap [b,B,\nu] \neq \varnothing.\] We see immediately that \(\lambda t\,W^\phi_S(\bar{x},t)\) is in fact \((a,A,\mu)\)-local and then it is straightforward to check that \((a,A,\mu) \forces \phi_S(\lambda t\,W^\phi_S(\bar{x},t);\bar{x}).\)
\end{proof}

\subsection{Witnessing in \ACA}\label{S:Witnessing:ACA}

Our second main result of this section is that all sentences of the forcing language admit witnessing in \ACA, but only in a weak form as in Corollary~\ref{C:Sigma3Witnessing}.

\begin{theorem}[\ACA]\label{T:Pi3Witnessing}%\label{T:ArithmeticWitnessing}
  If \(\theta\) is an \((a,A,\mu)\)-local \(\Pi^0_3\) sentence of the forcing language such that \((a,A,\mu) \forces \theta,\) then there is an extension \((b,B,\nu) \leq (a,A,\mu)\) such that \((b,B,\nu) \Sforces \theta.\)
\end{theorem}

\noindent
Again, the main application of Theorem~\ref{T:Pi3Witnessing} is that \(\Sigma^0_2\) formulas of the forcing language can be uniformized by names.

\begin{corollary}[\ACA]\label{C:Sigma2Uniformization}%\label{C:ArithmeticUniformization}
  If \(\phi(\bar{v},w)\) is an \((a,A,\mu)\)-local \(\Sigma^0_2\) formula of the forcing language such that \((a,A,\mu) \forces \forall\bar{v}\,\exists w\,\phi(\bar{v},w),\) then there are an extension \((b,B,\nu) \leq (a,A,\mu)\) and a \((b,B,\nu)\)-local name \(F(\bar{v})\) such that \((b,B,\nu) \forces \forall\bar{v}\,\phi(\bar{v},F(\bar{v})).\)
\end{corollary}

\noindent
By a bootstrapping process, we can obtain analogues of Theorem~\ref{T:Pi3Witnessing} and Corollary~\ref{C:Sigma2Uniformization} for all arithmetical formulas. However, the above suffices to show that \(F_\sigma\)-Mathias forcing preserves \ACA, so we will not push this further.

Before we get started with witnessing \(\Pi^0_3\) sentences, we will motivate one definition that would otherwise feel out of the blue. Consider the \((a,A,\mu)\)-local sentence \(\exists w\,\phi(w).\) If \((a,A,\mu) \Sforces \exists v\,\phi(v)\)  then, in particular, \((a,A,\mu) \forces \phi(F)\) for some \((a,A,\mu)\)-local nullary name \(F.\) We can then find \(\tau \in \tree{a,A}\) and \(y \in \N\) such that \((\tau,y) \in F\) and then \((a',A',\mu) \forces \phi(y)\) where \(a' = a \cup \tau^{-1}(1)\) and \(A' = A-\tau^{-1}(0).\) So the upshot of Skolemized forcing is that it allows for explicit witnessing of existential statements by making only finite changes to the original condition (and no change at all to the submeasure part). To analyze these finite changes, we introduce the notion of approximate forcing.

\begin{definition}
  Let \(\phi(w)\) be an \((a,A,\mu)\)-local \(\Pi^0_1\) formula of the forcing language. We define we define the \emph{approximate forcing relation} \((a,A,\mu) \starforces \exists w\,\phi(w)\) to hold if there are a \(y \in \N\) and an \(a' \in \fin{\N}\) such that \((a,A-a',\mu) \forces \phi(y).\)
\end{definition}

\noindent
Thus the approximate forcing relation \({\starforces}\) takes care of making finite changes to the infinite part of a condition. So, returning to the discussion that motivated this definition, if \((a,A,\mu) \Sforces \exists w\,\phi(w)\) then there are \(a \subseteq a' \subseteq A\) such that \((a',A,\mu) \starforces \exists w\,\phi(w).\)

The following lemma, which is key in most arguments using approximate forcing, is intended to be used in countable coded $\omega$-models, hence the weaker hypotheses.

\begin{lemma}[\WKL\ + \Bnd{\Sigma^0_2}]\label{L:StarSubmeasure}
  If \(\phi(\bar{v},w)\) is a \(\Pi^0_1\) formula of the forcing language, then there is a sequence \(\seq{\lambda_{a,\bar{x}}: a \in \fin{N},\bar{x} \in \N}\) of integer-valued submeasures such that the following statements hold for all conditions \((a,A,\mu)\) and all \(\bar{x} \in \N.\)
  \begin{itemize}
  \item If \(\phi(\bar{x},w)\) is \((a,A,\mu)\)-local and \((\mu\wedge\lambda_{a,\bar{x}})(A) < \infty\) then there is an extension \((a,A',\mu) \leq (a,A,\mu)\) such that \((a,A',\mu) \starforces \exists w\,\phi(\bar{x},w).\)
  \item If \(\phi(\bar{x},w)\) is \((a,A,\mu)\)-local and \((\mu\wedge\lambda_{a,\bar{x}})(A) = \infty\) then \((a,A,\mu) \nstarforces \exists w\,\phi(\bar{x},w).\)
  \end{itemize}
\end{lemma}

\begin{proof}
  Let \(\widehat{\phi}(a,A;\bar{v},w)\) be as in Proposition~\ref{P:Pi1Forcing}. Using the uniform procedure of Proposition~\ref{P:Submeasure}, compute an integer-valued submeasure \(\lambda_{a,\bar{x}}\) so that \(\FID(\lambda_{a,\bar{x}})\) is the smallest free ideal that contains the monotone \(F_\sigma\)-class \[\mathcal{F}_{a,\bar{x}} = \set{A \subseteq \N : \exists y\,\widehat{\phi}(a,A \cup a;\bar{x},y)}.\] 
Thus, in particular, if \(\phi(\bar{x},w)\) is \((a,A,\mu)\)-local and \((a,A,\mu) \starforces \exists w\,\phi(\bar{x},w),\) then \(\lambda_{a,\bar{x}}(A) < \infty.\) The second statement follows immediately from this observation. 

For the first statement, note that if \(\phi(\bar{x},w)\) is \((a,A,\mu)\)-local and \((\mu\wedge\lambda_{a,\bar{x}})(A) < \infty\) then there are \(A_0 \in \FID(\mu)\) and \(A_1,\dots,A_k \in \mathcal{F}_{a,\bar{x}}\) (\(k \in \N\)) such that \[A = A_0 \cup A_1 \cup \cdots \cup A_k.\] Since \(\mu(A) = \infty,\) there must be an \(i \in \set{1,\dots,k}\) such that \(\mu(A_i) = \infty.\) Then \((a,A_i \cup a,\mu) \leq (a,A,\mu)\) is a condition and certainly \((a,A_i \cup a,\mu) \starforces \exists w\,\phi(\bar{x},w),\) as required.
\end{proof}

In a set theoretic context, the next result would be the final part in showing that \(F_\sigma\)-Mathias forcing satisfies Axiom~A.~\cite{Baumgartner} However, because of our context, the statement and proof are quite different from its set theoretic counterpart.

\begin{lemma}[\ACA]\label{L:StarFusion}
  Let \(\phi(\bar{v},w)\) be an \((a_0,A_0,\mu_0)\)-local \(\Pi^0_1\) formula of the forcing language. There is a condition \((a_0,A,\mu) \leq (a_0,A_0,\mu_0)\) such that, for all \((b,B,\nu) \leq (a_0,A,\mu)\) and all \(\bar{x} \in \N,\) we have \((b,B,\nu) \starforces \exists w\,\phi(\bar{x},w) \IFF (b,A,\mu) \starforces \exists w\,\phi(\bar{x},w).\)
\end{lemma}

\begin{proof}
  Let \(M\) be a countable coded strict \(\beta\)-submodel (hence \(M \models \WKL\)) that contains \((a_0,A_0,\mu_0)\) and all names that occur in \(\phi.\) \cite[Theorem~VIII.2.11]{Simpson} We will define by induction a fusion sequence \(\seq{a_s,A_s,\mu_s}_{s=0}^\infty\) of conditions each of which is in \(M\) (but the whole sequence is not necessarily in \(M\)).

  Work inside \(M\) to set things up. Let \(\seq{b_s,\bar{x}_s}_{s=0}^\infty\) be an enumeration all tuples \(b,\bar{x}\) with \(b \in \fin{\N}\) and \(\bar{x} \in \N,\) with each tuple repeated infinitely often. For each \(s \in \N,\) let \(\lambda_s = \lambda_{b_s,\bar{x}_s}\) where \(\lambda_{b_s,\bar{x}_s}\) is as in Lemma~\ref{L:StarSubmeasure} for \(\phi(\bar{v},w).\)

  Working outside \(M.\) At stage \(s,\) define the condition \((a_{s+1},A_{s+1},\mu_{s+1}) \in M\) as follows.
  \begin{description}
  \item[\normalfont\emph{Case \(a_0 \nsubseteq b_s\) or \(b_s \nsubseteq a_s\)}:]
    Define \(\mu_{s+1} = \mu_s\) and \(A_{s+1} = A_s.\)
  \item[\normalfont\emph{Case \(a_0 \subseteq b_s \subseteq a_s\) and \((\mu_s\wedge\lambda_s)(A_s) = \infty\)}:]
    Define \(\mu_{s+1} = \mu_s\wedge(\lambda_s\vee s)\) and \(A_{s+1} = A_s.\)
  \item[\normalfont\emph{Case \(a_0 \subseteq b_s \subseteq a_s\) and \((\mu_s\wedge\lambda_s)(A_s) < \infty\)}:]
    Define \(\mu_{s+1} = \mu_s\) and, by applying Lemma~\ref{L:StarSubmeasure} in \(M,\) pick \(a_s \subseteq A_{s+1} \subseteq A_s\) in \(M\) such that \((b_s,A_{s+1},\mu_{s+1}) \starforces \exists w\,\phi(\bar{x}_s,w).\)
  \end{description}
  Finally, pick \(a_s \subseteq a_{s+1} \subseteq A_{s+1}\) such that \(\mu_{s+1}(a_{s+1}) \geq s+1\) so that \((a_{s+1},A_{s+1},\mu_{s+1}) \leq_{s+1} (a_s,A_s,\mu_s).\)

  Once the sequence has been constructed, let \(\mu = \bigwedge_{s=0}^\infty \mu_s\) and \(A = \bigcup_{s=1}^\infty a_s\) as per the Fusion Lemma. Thus \((a_0,A,\mu)\) is an extension of \((a_0,A_0,\mu_0).\)

  Suppose \((b,B,\nu) \leq (a_0,A,\mu)\) is a condition such that \((b,B,\nu) \starforces \exists w\,\phi(\bar{x},w).\) Choose \(s\) so that \(b_s \subseteq a_s\) and \(\bar{x} = \bar{x}_s.\) Then \(\lambda_s(B) = \lambda_{b,\bar{x}}(B) < \infty\) by Lemma~\ref{L:StarSubmeasure} and the definition of \(\lambda_s.\) Since \(\mu(B) = \nu(B) = \infty,\) we must have fallen into the third case at stage \(s\) of the construction. Thus \((b,A_{s+1},\mu_{s+1}) \starforces \exists w\,\phi(\bar{x},w)\) and, since \((b,A,\mu) \leq (b,A_{s+1},\mu_{s+1}),\) we see that \((b,A,\mu) \starforces \exists w\,\phi(\bar{x},w).\)
\end{proof}

We finally arrive at the proof of Theorem~\ref{T:Pi3Witnessing}. As before, we actually prove a stronger version which shows that the witnessing names can be chosen uniformly.

\begin{proposition}[\ACA]\label{P:Pi3Witnessing}
  Let \(\theta(\bar{v})\) be an \((a_0,A_0,\mu_0)\)-local \(\Pi^0_3\) formula of the forcing language. There are an extension \((a_0,A,\mu) \leq (a_0,A_0,\mu_0)\) and a partial name \(W^\theta_S(\bar{v},t)\) such that, for all extensions \((b,B,\nu) \leq (a_0,A,\mu)\) and all \(\bar{x} \in \N,\) if \((b,B,\nu) \forces \theta(\bar{x})\) then \(\lambda t\,W^\theta_S(\bar{x},t)\) is \((b,B,\nu)\)-local and \((b,B,\nu) \forces \theta_S(\lambda t\,W^\theta_S(\bar{x},t);\bar{x}).\)
\end{proposition}

\begin{proof}
  First note that if \((a_0,A,\mu)\) and \(W^\theta_S(\bar{v},w,t)\) are as required for the formula \(\theta(\bar{v},w),\) then \((a_0,A,\mu)\) and \(W^\psi_S(\bar{v},t) = W^\theta_S(\bar{v},\fst{t},\snd{t})\) are as required for the formula \(\forall w\,\theta(\bar{v},w).\) Therefore, it suffices to handle the case when \(\theta(\bar{v})\) is a \((a_0,A_0,\mu_0)\)-local \(\Sigma^0_2\) formula.

  Suppose \(\theta(\bar{v}) \equiv \exists w\,\phi(\bar{v},w),\) where \(\phi(\bar{v},w)\) is \(\Pi^0_1.\) Let \((a_0,A,\mu) \leq (a_0,A_0,\mu_0)\) be as in Lemma~\ref{L:StarFusion}. We proceed to define the partial name \(W^\theta_S(\bar{v},t).\) First consider the relation \[R(\tau,\bar{v},w) \IFF \tau \in \tree{a_0,A} \land (a_0 \cup \tau^{-1}(1), A-\tau^{-1}(0), \mu) \forces \phi(\bar{v},w).\] This is a \(\Pi^0_1\) relation by Proposition~\ref{P:Pi1Forcing}. Let \(\seq{\tau_i,y_i}_{i=0}^\infty\) be an enumeration of \(\tree{a_0,A}\times\N\) such that if \(\tau_i \subseteq \tau_j\) and \(y_i \leq y_j\) then \(i \leq j.\) Define the partial name \(F\) by \((\tau,\bar{x},y) \in F\) iff \(y = y_i,\) where \(i\) is minimal with the property that \(\tau_i \subseteq \tau\) and \(R(\tau_i,\bar{x},y_i).\) (Given our choice of enumeration, it is easy to check that this is indeed a partial name.) Finally, define the partial name \(W^\theta_S(\bar{v},t)\) by \[W^\theta_S(\bar{v},t) = \begin{cases} F(\bar{v}) & \text{when $t = 0$,} \\ W^\phi_S(\bar{v},F(\bar{v}),t-1) & \text{when $t \geq 1$,} \end{cases}\] where \(W^\phi_S(\bar{v},w,t)\) is as in Proposition~\ref{P:Pi1Witnessing}. 

  It is clear that if \((b,B,\nu) \leq (a_0,A,\mu)\) and \(\bar{x} \in \N\) are such that \(\lambda t\,W^\theta_S(\bar{x},t)\) is \((b,B,\nu)\)-local, then \((b,B,\nu) \forces \theta_S(\lambda t\,W^\theta_S(\bar{x},t);\bar{x}).\) Therefore, it suffices to show that if \((b,B,\nu) \forces \theta(\bar{x})\) then \(\lambda t\,W^\theta_S(\bar{x},t)\) is \((b,B,\nu)\)-local. Given Proposition~\ref{P:Pi1Witnessing}, we only need to show that \(\nbhd{b,B,\nu} \cap \dom F(\bar{x}) \neq \varnothing\) (where \(F(\bar{x})\) is considered as a nullary name).

  Since \((b,B,\nu) \forces \theta(\bar{x})\) we can find \((c,C,\kappa) \leq (b,B,\nu)\) and \(y \in \N\) such that \((c,C,\kappa) \forces \phi(\bar{x},y).\) By Lemma~\ref{L:StarFusion}, we then have \((c,A,\mu) \starforces \exists w\,\phi(\bar{x},w).\) Thus, there are a \(\tau \in \tree{a_0,A}\) and a \(y' \in \N\) such that \(c = \tau^{-1}(1)\) and \((\tau,\bar{x},y') \in F,\) which means that \(C - \tau^{-1}(0) \in \nbhd{b,B,\nu} \cap \dom F(\bar{x}).\)  
\end{proof}

The last proposition is all that is needed to show that \(F_\sigma\)-Mathias forcing preserves \ACA, but the same process can be continued through the entire arithmetical hierarchy. We will not prove the following proposition since it can essentially be deduced from Theorem~\ref{T:PresACA} and the fact that \ACA\ proves the existence of Skolem functions for arithmetical facts.

\begin{proposition}[\ACA]\label{P:AllWitnessing}
  Let \(\theta(\bar{v})\) be an \((a_0,A_0,\mu_0)\)-local formula of the forcing language. There are an extension \((a_0,A,\mu) \leq (a_0,A_0,\mu_0)\) and a partial name \(W^\theta_S(\bar{v},t)\) such that, for all extensions \((b,B,\nu) \leq (a_0,A,\mu)\) and all \(\bar{x} \in \N,\) if \((b,B,\nu) \forces \theta(\bar{x})\) then \(\lambda t\,W^\theta_S(\bar{x},t)\) is \((b,B,\nu)\)-local and \((b,B,\nu) \forces \theta_S(\lambda t\,W^\theta_S(\bar{x},t);\bar{x}).\)
\end{proposition}

\noindent
The proof of Proposition~\ref{P:AllWitnessing} is by induction on the complexity of \(\theta(\bar{v}).\) The existential quantifier steps are handled as we did for \(\Sigma^0_2\) formulas, except that uses of Proposition~\ref{P:Pi1Forcing} should be replaced by Proposition~\ref{P:Pi1Skolem}.

\section{The Generic Extension}\label{S:Extension}

The reader is invited to temporarily step out of the ground model into the ambient meta-world. Although not strictly necessary, it helps to think that the ground model is countable so that all generic objects discussed below can be proved to exist. To avoid unnecessary fuss, the reader can hold on to the belief that the ground model is an \(\omega\)-model. Our language will be tailored to this point of view, but all that is said will continue to hold true (perhaps vacuously) even in the worst case an uncountable non-standard ground model. So let us fix a generic filter \(\mathcal{G}\) for \(F_\sigma\)-Mathias forcing over our ground model \(\MN.\)

We first verify that \(F_\sigma\)-Mathias forcing is indeed a \emph{real forcing}, i.e.\ the generic filter \(\mathcal{G}\) is completely determined by the single real \[G = \bigcup \set{ a \in \fin{\N} : (a,A,\mu) \in \mathcal{G}}.\] This is the \emph{generic \(F_\sigma\)-Mathias real} associated to \(\mathcal{G}.\)

\begin{proposition}[\WKL]\label{P:GenericReal}
  \((a,A,\mu) \in \mathcal{G}\) if and only if \(a \subseteq G \subseteq A\) and \(\mu(G) = \infty.\)
\end{proposition}

\begin{proof}
  First, observe that if \(\nu\) is a submeasure coded in \(\MN,\) then \(\nu(G) < \infty\) if and only if there is a condition \((a,A,\mu) \in \mathcal{G}\) such that \(\nu(A) < \infty.\) The backward direction is clear. For the forward direction, suppose that \(\nu(G) \leq n\) and consider the class \[\mathcal{D} = \set{(a,A,\mu) : \nu(a) > n}.\] Since \(\mathcal{D} \cap \mathcal{G} = \varnothing,\) there must be a condition \((a,A,\mu) \in \mathcal{G}\) which has no extension in \(\mathcal{D}.\) This last statement is equivalent to \(\nu(A) \leq n.\)

  Suppose that \((a,A,\mu)\) is a condition such that \(a \subseteq G \subseteq A\) and \(\mu(G) = \infty.\) Since \(\mathcal{G}\) is a maximal filter, it suffices to show that \((a,A,\mu)\) is compatible with every condition \((b,B,\nu) \in \mathcal{G}.\) Since \(a \cup b \subseteq G \subseteq A \cap B,\) the only way in which \((a,A,\mu)\) and \((b,B,\nu)\) could be incompatible is that \((\mu\wedge\nu)(A \cap B) < \infty.\) So suppose \((b,B,\nu) \in \mathcal{G}\) is such that \((\mu\wedge\nu)(A\cap B) < \infty,\) then we can find disjoint \(A',B'\) in the ground model \(\N\) such that \(A \cap B = A' \cup B'\) and \(\mu(A') < \infty,\) \(\nu(B') < \infty.\) Consider the class \[\mathcal{C} = \set{ (c,C,\kappa) : C - c \subseteq A' }.\] Since \(\mathcal{C}\cap\mathcal{G} = \varnothing\) there must be a \((d,D,\lambda) \in \mathcal{G}\) with no extension in \(\mathcal{C},\) i.e., \(\lambda(D\cap A') < \infty.\) Without loss of generality, \((d,D,\lambda) \leq (b,B,\nu)\) which means that \(\lambda(D \cap B') < \infty\) too. But then \(G \subseteq A\cap B \cap D = (A' \cup B')\cap D\) and so \[\lambda(G) \leq \lambda(A' \cap D)+\lambda(B' \cap D) < \infty,\] which is impossible. Therefore \((a,A,\mu) \in \mathcal{G},\) as required.
\end{proof}

\noindent
In view of the above, we will now forget about the generic filter \(\mathcal{G}\) and work only with the generic real \(G.\) Instead of writing \((a,A,\mu) \in \mathcal{G},\) we will simply say that the condition \((a,A,\mu)\) is compatible with \(G.\)

Next, we verify that the generic extension \(\MN[G]\) is well defined. A name \(F\) is \emph{\(G\)-local} if and only if it is \((a,A,\mu)\)-local for some condition \((a,A,\mu)\) compatible with \(G.\) A formula \(\phi\) of the forcing language is \emph{\(G\)-local} if and only if every name that occurs in \(\phi\) is \(G\)-local.

\begin{proposition}[\WKL]\label{P:GenericFun}
  If the name \(F\) is \(G\)-local, then \[F^G(\bar{x}) = y \IFF \exists n\,{(G \res n,\bar{x},y) \in F}\] defines a total \(k\)-ary function.
\end{proposition}

\begin{proof}
  Suppose \(F\) is a \(G\)-local name. Fix \(\bar{x} \in \N\) and consider the class \[\mathcal{F}_{\bar{x}} = \set{(a,A,\mu):\exists y,\tau\,(\tau^{-1}(1) = a \land (\tau,\bar{x},y) \in F)}.\] If no element of \(\mathcal{F}_{\bar{x}}\) is compatible with \(G,\) then there must be a condition \((b,B,\nu)\) compatible with \(G\) which has no extension in \(\mathcal{F}_{\bar{x}}.\) In that case, \(F\) is not \(G\)-local since \(F\) is not \((c,C,\kappa)\)-local for any \((c,C,\kappa) \leq (b,B,\nu).\)
\end{proof}

\noindent
The generic extension \(\MN[G]\) is the functional model with the same base set \(\N\) and where the \(k\)-ary functions are \[\FN_k[G] = \set{F^G : \mbox{\(F\) is a \(G\)-local \(k\)-ary name}}.\] The observation made in Section~\ref{S:Forcing} that composition and primitive recursion preserve locality ensures that \(\MN[G]\) satisfies all of our basic axioms. The next two theorems show that uniformization and minimization are preserved in the generic extension.

\begin{theorem}[\ACA]\label{T:PresACA}
  If \(G\) is generic for \(F_\sigma\)-Mathias forcing, then \(\MN[G] \models \ACA.\)
\end{theorem}

\noindent
We do not know whether \WKL\ + \Bnd{\Sigma^0_2} is preserved by \(F_\sigma\)-Mathias forcing, but \WKL\ + \Ind{\Sigma^0_2} is preserved.

\begin{theorem}[\WKL + \Ind{\Sigma^0_2}]\label{T:PresIS2}
  If \(G\) is generic for \(F_\sigma\)-Mathias forcing, then \(\MN[G] \models \WKL + \Ind{\Sigma^0_2}.\)
\end{theorem}

To prove these preservation theorems, we will establish a series of results that relate truth in the generic extension \(\MN[G]\) and the forcing relation in the ground model \(\MN.\) As usual, the keystone is the \(\Pi^0_1\) case.

\begin{proposition}[\WKL]\label{P:Pi1Truth}
  If \(\phi\) is a \(\Pi^0_1\) sentence of the forcing language and there is a condition \((a,A,\mu)\) compatible with \(G\) such that \(\phi\) is \((a,A,\mu)\)-local and \((a,A,\mu) \forces \phi,\) then \(\MN[G] \models \phi^G.\)
\end{proposition}

\begin{proof}
  By Proposition~\ref{P:BoundedForcing}, we may assume that \(\phi\) is of the form \(\forall\bar{v}\,{T(\bar{v}) = 0},\) where \(T\) is \((a,A,\mu)\)-local. Since \((a,A,\mu) \forces \phi,\) we know that \((G\res n,\bar{x},y) \in T\) implies \(y = 0.\) Since \(T^G\) is total, we conclude that \(T^G(\bar{x}) = 0\) for every \(\bar{x} \in \N.\) By definition of \(\MN[G],\) we see that \(\MN[G] \models \forall\bar{v}\,{T(\bar{v}) = 0}.\)
\end{proof}

\noindent
Since the Skolemization \(\phi_S(W)\) is always equivalent to a \(\Pi^0_1\) formula, we have a useful corollary to this last proposition.

\begin{corollary}[\WKL]\label{C:SkolemTruth}
  If \(\phi\) is a \(G\)-local sentence and there is a condition \((a,A,\mu)\) compatible with \(G\) such that \((a,A,\mu) \Sforces \phi,\) then \(\MN[G] \models \phi^G.\)
\end{corollary}

The following result is the key for proving that \(F_\sigma\)-Mathias forcing preserves \RCA\ over \WKL\ + \Bnd{\Sigma^0_2}.

\begin{proposition}[\WKL\ + \Bnd{\Sigma^0_2}]\label{P:Pi2Truth}
  If \(\phi\) is a \(G\)-local \(\Pi^0_2\) sentence of the forcing language, then \(\MN[G] \models \phi\) if and only if there is a condition \((a,A,\mu)\) compatible with \(G\) such that \((a,A,\mu) \Sforces \phi.\)
\end{proposition}

\begin{proof}
  Suppose that \(\phi\) is \((a_0,A_0,\mu_0)\)-local, where \((a_0,A_0,\mu_0)\) is compatible with \(G,\) and write \(\phi \equiv \forall w\,\lnot\psi(w),\) where \(\psi(w)\) is \(\Pi^0_1.\) Let \(\varrho_0\) be as in Proposition~\ref{P:Pi2Decision} and consider the class \[\mathcal{D} = \set{(a,A,\mu) \leq (a_0,A_0,\mu_0) : \mu \leq \varrho_0 \lor \exists y\,\widehat{\psi}(a,A;y)}.\] Note that \(\mathcal{D}\) is dense below \((a_0,A_0,\mu_0),\) so there is a condition \((a,A,\mu) \in \mathcal{D}\) which is compatible with \(G.\) By Propositions~\ref{P:Pi2Decision} and~\ref{P:Pi2Witnessing}, if \(\mu \leq \varrho_0\) then \((a,A,\mu) \Sforces \phi,\) and if \(\widehat{\psi}(a,A;y)\) then \((a,A,\mu) \Sforces \lnot\phi.\) In either case, the result follows from Proposition~\ref{C:SkolemTruth}.
\end{proof}

\begin{proposition}[\WKL\ + \Bnd{\Sigma^0_2}]\label{P:PresRCA}
  If \(G\) is generic for \(F_\sigma\)-Mathias forcing, then \(\MN[G] \models \RCA.\)
\end{proposition}

\begin{proof}
  Let \(H\) be a \((k+1)\)-ary name such that \[\MN[G] \models \forall\bar{x}\,\exists y\,{H^G(\bar{x},y) = 0}.\] Since this \(\Pi^0_2\) statement is true, it follows from Proposition~\ref{P:Pi2Truth} that there is a condition \((a,A,\mu)\) compatible with \(G\) such that \[(a,A,\mu) \forces \forall\bar{x}\,\exists y\,{H(\bar{x},y) = 0}.\] Then, by Corollary~\ref{C:Sigma1Uniformization}, there is a \(k\)-ary \((a,A,\mu)\)-local name \(F\) such that \[(a,A,\mu) \forces \forall\bar{x},z\,{H(\bar{x},F(\bar{x})) = 0}.\] It follows from Proposition~\ref{P:Pi1Truth} that \(\MN[G] \models \forall\bar{x}\,{H^G(\bar{x},F^G(\bar{x})) = 0}.\)
\end{proof}

The following result is key to prove Theorem~\ref{T:PresACA}.

\begin{proposition}[\ACA]\label{P:Pi3Truth}
  If \(\phi\) is a \(G\)-local \(\Pi^0_3\) sentence of the forcing language, then \(\MN[G] \models \phi^G\) if and only if there is a condition \((a,A,\mu)\) compatible with \(G\) such that \((a,A,\mu) \Sforces \phi.\)
\end{proposition}

\begin{proof}
  By Corollary~\ref{C:Sigma3Witnessing} and Proposition~\ref{P:Pi3Witnessing}, we know that \[\set{(a,A,\mu) \leq (a_0,A_0,\mu_0) : (a,A,\mu) \Sforces \phi \lor (a,A,\mu) \Sforces \lnot\phi}\] is dense below \((a_0,A_0,\mu_0).\) However, this is not a \(\Sigma^1_1\) class since \({\Sforces}\) hides an implicit quantification over all \((a,A,\mu)\)-local names which form a \(\Pi^1_1\) class. To remedy this, we consider the class \(\mathcal{D}\) of conditions \((a,A,\mu) \leq (a_0,A_0,\mu_0)\) for which there is a partial unary name \(W\) such that \[\mu \leq \bigwedge_{a \subseteq b \subseteq A} (\vartheta_b \vee |b|)\] and either \((a,A,\mu) \forces \phi_S(W)\) or \((a,A,\mu) \forces \lnot\phi_H(W),\) where \(\seq{\vartheta_b:b \in \fin{\N}}\) is as in Proposition~\ref{P:FunTotal} for the name \(W.\) This is a \(\Sigma^1_1\) class and it too is dense below \((a_0,A_0,\mu_0)\). Therefore, we can find a condition \((a,A,\mu) \in \mathcal{D}\) which is compatible with \(G.\) Since \(\mu \leq \tau_W\) implies that \(W\) is \((a,A,\mu)\)-local, it follows that either \((a,A,\mu) \Sforces \phi\) or \((a,A,\mu) \Sforces \lnot\phi\) and the result follows from Corollary~\ref{C:SkolemTruth}.
\end{proof}

\begin{proof}[Proof of Theorem~\ref{T:PresACA}]
  Let \(H\) be a \((k+1)\)-ary name. Since \(\MN[G]\) satisfies \Ind{\Sigma^0_1}, we see that \[\MN[G] \models \forall\bar{x}\,\exists y\,\forall z\,{H^G(\bar{x},y) \leq H^G(\bar{x},z)}.\] Therefore, by Proposition~\ref{P:Pi3Truth}, there is a condition \((a,A,\mu)\) compatible with \(G\) such that \[(a,A,\mu) \Sforces \forall\bar{x}\,\exists y\,\forall z\,H(\bar{x},y) \leq H(\bar{x},z).\] Then, as in Corollary~\ref{C:Sigma2Uniformization}, there is a \(k\)-ary name \(F\) such that \[(a,A,\mu) \forces \forall\bar{x},z\,{H(\bar{x},F(\bar{x})) \leq H(\bar{x},z)}.\] It follows from Proposition~\ref{P:Pi1Truth} that \[\MN[G] \models \forall\bar{x},z\,{H^G(\bar{x},F^G(\bar{x})) \leq H^G(\bar{x},z)}.\qedhere\]
\end{proof}

Before we prove Theorem~\ref{T:PresIS2}, we show that \WKL\ is preserved under the weaker assumption of \WKL\ + \Bnd{\Sigma^0_2}.

\begin{proposition}[\WKL\ + \Bnd{\Sigma^0_2}]\label{P:PresWKL}
  If \(G\) is generic for \(F_\sigma\)-Mathias forcing, then \(\MN[G] \models \WKL.\)
\end{proposition}

\begin{proof}
  Since we already know that \(\MN[G] \models \RCA,\) it is enough to check that every infinite subtree of \(2^{<\infty}\) in \(\MN[G]\) has a branch. Let \(S\) be a \(G\)-local name for an infinite subtree of \(2^{<\infty}.\) Specifically, let \(S\) be a unary name such that each \(S(n)\) is forced to be (a code for) the \(n\)-th level of the tree in question. (We will tacitly identify the code with the coded level set.)

  For each \(n,\) let \(\mathcal{B}_n\) be the closed class of all partial names \(B\) such that if \(\tau \in 2^{\leq n},\) \(m \leq n,\) and \(x \subseteq 2^m\) are such that \((\tau,m,x) \in S\) then there is a \(\sigma \in x\) such that \((\tau,i,\sigma(i)) \in B\) for every \(i \leq m.\) Each \(\mathcal{B}_n\) is nonempty and so, by \WKL, the intersection \(\mathcal{B} = \bigcap_{n=0}^\infty \mathcal{B}_n\) is nonempty too. If \(B \in \mathcal{B}\) then \(B\) is a partial name with the same domain as \(S.\) Thus, \(B\) is also \(G\)-local and \(\MN[G] \models \forall n\,{B^G\res n \in S^G(n)}.\) 
\end{proof}
 
\begin{proof}[Proof of Theorem~\ref{T:PresIS2}]
  Since we already know that \(\MN[G] \models \WKL,\) it is enough to check that \(\MN[G] \models \Ind\Sigma^0_2.\) Let \(\phi(u,v)\) be a \(G\)-local \(\Pi^0_1\) formula of the forcing language. Let \(\seq{\lambda_{a,x}: a \in \fin{N},x,y \in \N}\) be as in Lemma~\ref{L:StarSubmeasure} for \(\phi(u,v).\)

  Let \((a_0,A_0,\mu_0)\) be any condition compatible with \(G\) such that \(\phi(u,v)\) is \((a_0,A_0,\mu_0)\)-local and \((a_0,A_0,\mu_0) \forces \phi(x_0,y_0)\) for some \(x_0,y_0 \in \N.\) Then, for each $x \in \N$ define \[\nu_x = \bigwedge_{a_0 \subseteq a \subseteq A_0} \lambda_{a,x} \vee |a|.\] By \Ind{\Sigma^0_2} there is a minimal \(x \leq x_0\) such that \[(\mu_0\wedge\nu_0\wedge\cdots\wedge\nu_x)(A_0) < \infty.\] 

  Let \(\mu = \mu_0\wedge\bigwedge_{z<x} \nu_z.\) Then, by Lemma~\ref{L:StarSubmeasure}, \(\mu(A_0) = \infty\) and \[(a_0,A_0,\mu) \forces \forall u < x\,\forall v\,{\lnot\phi(u,v)}.\] Also, there are an extension \((a,A,\mu) \leq (a_0,A_0,\mu)\) and a \(y \in \N\) such that \((a,A,\mu) \forces \phi(x,y).\) Therefore, \((a,A,\mu)\) forces that \(x\) is least such that \(\exists y\,\phi(x,y).\)

  Now consider the class \[\mathcal{D} = \set{(a,A,\mu) \leq (a_0,A_0,\mu_0) : \exists x,y\,(\mu \leq {\textstyle\bigwedge_{z<x} \nu_z^{a,A}} \land \widehat{\phi}(a,A;x,y))},\] where \(\nu_z^{a,A} = \bigwedge_{a \subseteq b \subseteq A} \lambda_{z,b} \vee |b|\) as above, and \(\widehat{\phi}(a,A;x,y)\) is as in Proposition~\ref{P:Pi1Forcing}. The above shows that \(\mathcal{D}\) is dense below \((a_0,A_0,\mu_0)\) and hence there is a condition \((a,A,\mu) \leq (a_0,A_0,\mu_0)\) compatible with \(G\) such that \((a,A,\mu) \in \mathcal{D}.\) Then \((a,A,\mu)\) forces that there is a minimal \(x\) such that \(\MN[G] \models \exists v\,\phi^G(x,v).\)
\end{proof}

\section{Applications}\label{S:Applications}

\subsection{Cohesive sets}

The following proposition clarifies the relation between \(F_\sigma\)-Mathias forcing and traditional Mathias forcing.

\begin{proposition}[\WKL\ + \Bnd{\Sigma^0_2}]\label{P:Traditional}
  If \(G\) is generic for \(F_\sigma\)-Mathias forcing, then \(G\) is \(\Pi^0_1\)-generic but not \(\Pi^0_2\)-generic for traditional Mathias forcing.
\end{proposition}

\begin{proof}
  The fact that \(G\) is \(\Pi^0_1\)-generic is an immediate consequence of Proposition~\ref{P:Pi1Forcing} since the formula \(\widehat{\phi}\) is independent of the submeasure component of the \(F_\sigma\)-Mathias condition. Thus the forcing relation for \(\Pi^0_1\) formulas of the forcing language is the same for \(F_\sigma\)-Mathias forcing as for traditional Mathias forcing.

  Given a function \(F:\N\to\N,\) let \(\mu_f\) be the submeasure such that \(\FID(\mu_f)\) is generated as in Proposition~\ref{P:Submeasure} by the closed class of all \(A \subseteq \N\) whose enumeration functions dominate \(F.\) Note that the \(\Sigma^1_1\) class \[\mathcal{F} = \set{(b,B,\nu) : \exists F\,{\nu \leq \mu_F}}\] is open dense. Indeed, given \((a,A,\mu),\) \(F(0) = 0\) and for each \(n \geq 1\) let \(F(n)\) be the first element of \(A\) such that \(\mu(A \cap \set{F(n-1),\dots,F(n)-1}) \geq n.\) Then \((\mu\wedge\mu_F)(A) = \infty\) and hence \((a,A,\mu\wedge\mu_F) \in \mathcal{F}.\)

  If \((b,B,\nu)\) is such that \(\nu \leq \mu_F\) and if \(E\) is the name for the enumeration function of the generic real, then \[(b,B,\nu) \forces \forall m\,\exists n\,(n \geq m \land E(n) < F(n))\] in \(F_\sigma\)-Mathias forcing, whereas \[(b,B) \forces \exists m\,\forall n\,(n \geq m \lthen E(n) > F(n))\] in traditional Mathias forcing.   
\end{proof}

\noindent
A similar argument shows that \(F_\sigma\)-Mathias forcing does not add a dominating real. However, \(\Pi^0_1\) genericity suffices to show that the generic real is either almost contained in or almost disjoint from any ground model subset of \(\N.\) By Theorem~\ref{T:PresIS2}, we can iterate \(F_\sigma\)-Mathias forcing over a model of \WKL\ + \Ind{\Sigma^0_2} to obtain a model of \WKL\ + \COH\ + \Ind{\Sigma^0_2}.

\begin{corollary}\emph{(Cholak--Jockusch--Slaman~\cite{CholakJockuschSlaman})}
  \textup{\WKL\ + \COH\ + \Ind{\Sigma^0_2}} is \(\Pi^1_1\)-conservative over \textup{\WKL\ + \Ind{\Sigma^0_2}}.
\end{corollary}

\subsection{Conservation results}

It was observed by Shore that forcing constructions give stronger conservation than one would normally expect for $\omega$-extensions. In the case of \(F_\sigma\)-Mathias forcing, this phenomenon takes the following form.

\begin{theorem}[\WKL\ + \Bnd{\Sigma^0_2}]\label{T:R12Conservativity}
  If \(G\) is generic for \(F_\sigma\)-Mathias forcing then \(\MN[G]\) is a conservative extension of \(\MN\) for \(\Pi^1_2\) statements of the form \[\forall X\,(\phi(X) \lthen \exists Y\psi(X,Y)),\] where \(\phi(X)\) is \(\Sigma^1_1\) and \(\psi(X,Y)\) is \(\Pi^0_2.\)
\end{theorem}

\begin{proof}
  Suppose that \(X \in \MN\) is such that \(\MN \models \phi(X)\) and \(\MN[G] \models \psi(X,Y^G)\) for some \(G\)-local name \(Y.\) Let \((a,A,\mu)\) be a condition compatible with \(G\) such that \(Y\) is \((a,A,\mu)\)-local and \((a,A,\mu) \forces \psi(\check{X},Y).\) By Theorem~\ref{T:Pi2Witnessing}, there is an \((a,A,\mu)\)-local name \(W\) such that \((a,A,\mu) \forces \psi_S(W,\check{X},Y).\) If \(H \in \nbhd{a,A,\mu} \cap \dom(W) \cap \dom(Y),\) then \(\MN \models \psi_S(W^H,X,Y^H)\) and hence \(\MN \models \psi(X,Y^H).\)  
\end{proof}

\noindent
Note that this is sharp since \COH\ has the above form where \(\psi(X,Y)\) is \(\Pi^0_3\) instead of \(\Pi^0_2.\)

Since the Kleene Normal Form Theorem holds in \ACA, the restricted form of Theorem~\ref{T:R12Conservativity} is no restriction at all.

\begin{corollary}[\ACA]\label{C:P12Conservativity}
  If \(G\) is generic for \(F_\sigma\)-Mathias forcing then \(\MN[G]\) is a conservative extension of \(\MN\) for \(\Pi^1_2\) statements.
\end{corollary}

\noindent
This is optimal since there are \(\Pi^1_3\) statements which are forced by iterated \(F_\sigma\)-Mathias forcing.

\subsection{Avoiding cones}

We will now show that the \(F_\sigma\)-Mathias generic can be forced to avoid cones. As an immediate consequence of Theorem~\ref{T:Cone}, we see that the \emph{needed reals}~\cite{Blass} for \(F_\sigma\)-Mathias forcing are precisely the computable reals. This is a sharp contrast with traditional Mathias forcing where the needed reals are precisely the hyperarithmetic reals. We will need the following fact, which is easily derivable using methods of Lawton (cf.~\cite[\S2]{DzhafarovJockusch}).

\begin{lemma}[\ACA]\label{L:Lawton}
  Let \(D = \seq{D_n}_{n=0}^\infty\) be a sequence of sets. For every infinite tree \(T,\) if \(D_n \nleq_T T\) for every \(n,\) then there is a branch \(B \in [T]\) such that \(D_n \nleq_T B\) and \(B' \leq_T T' \oplus D.\)
\end{lemma}

\begin{theorem}[\ACA]\label{T:Cone}
  Let \(D = \seq{D_n}_{n=0}^\infty\) be a sequence of non-computable sets. There is a condition \((\varnothing,A,\mu)\) such that \((\varnothing,A,\mu) \forces \forall n\,{D_n \nleq_T G}.\)
\end{theorem}

\begin{proof}
  Let \(\Phi_e\) denote the \(e\)-th Turing functional, which can be viewed as partial name. For each \(e\) and \(b \in \fin{\N},\) let \(\kappa_{b,e}\) be the submeasure such that \(\FID(\kappa_{b,e})\) is generated by the \(\Pi^0_1\) class \(\mathcal{C}_{e,b}\) of all \(C \subseteq \N\) such that \[\forall x, \tau_1,\tau_2\,(\tau_1,\tau_2 \in \tree{b,C \cup b} \land \Phi_e^{\tau_1}(x){\downarrow} \land \Phi_e^{\tau_2}(x){\downarrow} \lthen \Phi_e^{\tau_1}(x) = \Phi_e^{\tau_2}(x)).\] Let \(\seq{e_s,b_s}_{s=0}^\infty\) be an enumeration of \(\N\times\fin{\N}\) with infinite repetitions, and write \(\kappa_s\) and \(\mathcal{C}_s\) for \(\kappa_{e_s,b_s}\) and \(\mathcal{C}_{e_s,b_s},\) respectively.

  We define a fusion sequence of conditions \(\seq{a_s,A_s,\mu_s}_{s=0}^\infty\) such that, for each \(s,\) \(A_s' \oplus D \equiv_T \varnothing' \oplus D\) and \(A_s \nleq_T D_n\) for every \(n.\) First, define \((a_0,A_0,\mu_0) = (\varnothing,\N,\sup).\) Then, define the condition \((a_{s+1},A_{s+1},\mu_{s+1})\) as follows.
  \begin{description}
  \item[\normalfont\emph{Case \(a_0 \nsubseteq b_s\) or \(b_s \nsubseteq a_s\)}:]
    Define \(\mu_{s+1} = \mu_s\) and \(A_{s+1} = A_s.\)
  \item[\normalfont\emph{Case \(a_0 \subseteq b_s \subseteq a_s\) and \((\mu_s\wedge\kappa_s)(A_s) = \infty\)}:]
    Define \(\mu_{s+1} = \mu_s\wedge(\kappa_s\vee s)\) and \(A_{s+1} = A_s.\)
  \item[\normalfont\emph{Case \(a_0 \subseteq b_s \subseteq a_s\) and \((\mu_s\wedge\kappa_s)(A_s) < \infty\)}:]
    By applying Lemma~\ref{L:Lawton} to the tree of Lemma~\ref{L:Submeasure:2}, pick (an index for) \(C \in \mathcal{C}_s\cap\nbhd{b_s,A_s,\mu_s}\) such that \(C' \leq_T A_s' \oplus D \equiv_T \varnothing' \oplus D,\) and \(D_n \nleq_T C\) for all \(n.\) Then define \(\mu_{s+1} = \mu_s\) and \(A_{s+1} = C \cup a_s.\) 
  \end{description}
  Finally, pick \(a_s \subseteq a_{s+1} \subseteq A_{s+1}\) such that \(\mu_{s+1}(a_{s+1}) \geq s+1.\) Note that we invariably have \(A_{s+1}' \oplus D \equiv_T \varnothing' \oplus D\) and \(D_n \nleq_T A_{s+1}\) for every \(n.\)

  Once the sequence has been constructed, let \(\mu = \bigwedge_{s=0}^\infty \mu_s\) and \(A = \bigcup_{s=1}^\infty a_s.\) The condition \((\varnothing,A,\mu)\) is as required. To see this, suppose instead that \(\Phi_e^G = D_n,\) where \(G\) is a \(F_\sigma\)-Mathias generic real compatible with \((\varnothing,A,\mu).\) It follows that \(\Phi_e\) is a \(G\)-local name. Let \((b,B,\nu) \leq (\varnothing,A,\mu)\) be a condition compatible with \(G\) such that \(\Phi_e\) is \((b,B,\nu)\)-local and that \((b,B,\nu) \forces \forall x\,{\Phi_e^G(x) = D_n(x)}.\) Let \(s\) be large enough that \(e = e_s\) and \(b = b_s \subseteq a_s.\) Since \(B \in \mathcal{C}_s,\) we must have been in the case \((\mu_s\wedge\kappa_s)(A_s) < \infty\) of the construction. Without loss of generality, we have \(B \subseteq C,\) where \(C\) is the element of \(\mathcal{C}_s\) chosen at stage \(s.\) Now since \(\Phi_e\) is \((b,B,\nu\))-local, for every \(x\) there is a \(\tau \in U(b,B) \subseteq U(b,C)\) such that \(\Phi_e^\tau(x){\downarrow}.\) By choice of \(C,\) for any \(\tau \in U(b,C)\) if \(\Phi_e^\tau(x){\downarrow}\) then \(\Phi_e^\tau(x) = D_n(x).\) This means that \(D_n \leq_T C,\) which contradicts the fact that \(D_n \nleq_T C.\)
\end{proof}

\noindent
However, observe that for every set \(C\) there is a condition \((\varnothing,A,\sup) \forces C \leq_T G,\) namely let \(A = \set{\textstyle2^n + \sum_{i=0}^{n-1} C(i)2^i : n \in \N}.\)

\begin{corollary}[\ACA]\label{C:Cone}
  Let $D = \seq{D_n}_{n=0}^\infty$ be a sequence of non-computable sets. 
  \begin{enumerate}[\upshape(i)]
  \item\label{C:cone:coh} For every sequence \(R = \seq{R_n}_{n=0}^\infty\) there is a \(R\)-cohesive set \(G\) such that \(D_n \nleq_T G\) for every \(n.\)%\footnote{This result is due to Dzhafarov and Jockusch~\cite[Lemma~3.1]{DzhafarovJockusch} in the case where each \(R_n\) is computable.}
  \item\label{C:cone:par} \textup{(Seetapun~\cite{SeetapunSlaman}, \(k=2\); Dzhafarov--Jockusch~\cite{DzhafarovJockusch})} For every finite partition \(A_1,\dots,A_k\) of \(\N,\) one of the pieces \(A_i\) contains an infinite set \(G\) such that \(D_n \nleq_T G\) for every \(n.\)%\footnote{This result is due to Seetapun~\cite[Theorem~2.19]{SeetapunSlaman} in the case \(k=2\); the general case is due to Dzhafarov and Jockusch~\cite[Lemma~3.2(i)]{DzhafarovJockusch}}
  \item\label{C:cone:hom} \textup{(Seetapun~\cite{SeetapunSlaman})} For every computable coloring \(C:[\N]^2\to\set{1,\dots,k},\) there is an infinite \(C\)-homogeneous set \(H\) such that \(D_n \nleq_T H\) for every \(n.\)%\footnote{This result due to Seetapun~\cite[Theorem~2.1]{SeetapunSlaman} in the case where \(c:[\N]^2\to\set{1,\dots,k}\) is computable.}
  \end{enumerate}
\end{corollary}

\noindent
Note that there is no claim that the set \(G\) is generic. Indeed, the ground model \(\MN\) is a \(\beta\)-submodel of the generic extension \(\MN[G].\) Since every instance of~\eqref{C:cone:coh} and~\eqref{C:cone:par} is a \(\Sigma^1_1\) fact in \(\MN[G],\) it must already be true in \(\MN.\) Part~\eqref{C:cone:hom} is obtained by combining parts~\eqref{C:cone:coh}, to obtain a stable subcoloring of \(C,\) and~\eqref{C:cone:par}, to find an almost homogeneous subset for this stable subcoloring, from which one easily computes a homogeneous set for \(C.\)

\bibliographystyle{amsplain}
\bibliography{fsmathias}

\end{document}